\documentclass[10pt,oneside]{amsart} %added oneside so comments stay on the right, can remove when finished

\usepackage{graphicx} % Required for inserting images

\usepackage{amsmath,amstext,amscd, amssymb,txfonts, epsfig, color, epstopdf}
\usepackage{pinlabel}
\usepackage[normalem]{ulem}
\usepackage[colorlinks=true, pdfstartview=FitV, linkcolor=blue, citecolor=blue, urlcolor=blue]{hyperref}
 \usepackage[abs]{overpic}
\usepackage{xcolor,varwidth}
\usepackage{enumerate}
\usepackage{tabularx}
\usepackage{tikz}
 \usepackage{enumitem}
 \usepackage{csquotes}
\makeatletter
\def\@tocline#1#2#3#4#5#6#7{\relax
\ifnum #1>\c@tocdepth % then omitf
  \else 
    \par \addpenalty\@secpenalty\addvspace{#2}% 
\begingroup \hyphenpenalty\@M
    \@ifempty{#4}{%
      \@tempdima\csname r@tocindent\number#1\endcsname\relax
 }{%
   \@tempdima#4\relax
 }%
 \parindent\z@ \leftskip#3\relax \advance\leftskip\@tempdima\relax
 \rightskip\@pnumwidth plus4em \parfillskip-\@pnumwidth
 #5\leavevmode\hskip-\@tempdima #6\nobreak\relax
 \ifnum#1<0\hfill\else\dotfill\fi\hbox to\@pnumwidth{\@tocpagenum{#7}}\par
 \nobreak
 \endgroup
  \fi}
\makeatother
\let\oldtocsection=\tocsection
\let\oldtocsubsection=\tocsubsection
\let\oldtocsubsubsection=\tocsubsubsection
\renewcommand{\tocsection}[2]{\hspace{0em}\oldtocsection{#1}{#2}}
\renewcommand{\tocsubsection}[2]{\hspace{1em}\oldtocsubsection{#1}{#2}}
\renewcommand{\tocsubsubsection}[2]{\hspace{2em}\oldtocsubsubsection{#1}{#2}}
\definecolor{cerulean}{rgb}{0,.48,.65} 
\definecolor{magenta}{rgb}{.5,0,.5} 
\definecolor{dred}{rgb}{.5,0,0} 
\definecolor{green}{rgb}{0,.5,0} 
\definecolor{blue}{rgb}{0,0,1} 
\definecolor{black}{rgb}{0,0,0} 
\definecolor{dgreen}{rgb}{0,.3,0} 
\definecolor{vdred}{rgb}{.3,0,0} 
\definecolor{red}{rgb}{1,0,0} 
\definecolor{salmon}{rgb}{0.98,0.50,0.45} 
\definecolor{gray}{rgb}{.5,.5,.5} 
\definecolor{seagreen}{rgb}{0.13,0.70,0.67} 
\definecolor{chartreuse}{rgb}{0.40,0.80,0.00}
\definecolor{cornflower}{rgb}{0.39,0.58,0.93} 
\definecolor{gold}{rgb}{0.80,0.68,0.00}
 
\theoremstyle{plain}

\newtheorem{thm}{Theorem}[section]
\newtheorem{lemma}[thm]{Lemma}

\newtheorem{cor}[thm]{Corollary}
\newtheorem{prop}[thm]{Proposition}

\newtheorem{Question}[thm]{Question}

\theoremstyle{definition}
\newtheorem{rem}[thm]{Remark}
\newtheorem{defn}[thm]{Definition}

\newtheorem{remark}[thm]{Remark}

\newtheorem{Example}[thm]{Example}

\newtheorem{Open questions}[thm]{Open questions}
\newtheorem{Open question}[thm]{Open question}
\newtheorem{Open problems}[thm]{Open problems}
\newtheorem{Open problem}[thm]{Open problem}

\def\Bbb{\mathbb}

\def\Naturals{\Bbb{N}}

\def\ni{\noindent}

\def\CL{\hbox{\rm CL}}

\def\F+L{\hbox{$\textup{F}\!_+\textup{L}$}}

\newcommand{\BS}{\mathrm{BS}}

\def\onto{{\kern3pt\to\kern-8pt\to\kern3pt}}

\def\<{\langle}
\def\>{\rangle}
\def\|{{\ |\ }}

\newcommand{\abs}[1]{\left|#1\right|}
\renewcommand{\ni}{\noindent}

\def\*{^{\star}}

\newcommand{\NN}{\mathbb{N}}

\newcommand{\ZZ}{\mathbb{Z}}

\newcommand{\del}{\partial}

\newcommand{\rk}{\textup{rk}}

\title{Conjugator Length in the Baumslag-Gersten group}
\author{Conan Gillis}
\date{November 2024}

\begin{document}

\maketitle
\begin{abstract}
    We show that the conjugator length function of the Baumslag-Gersten group is equivalent to a tower of exponentials of height $\lfloor \log_2n\rfloor$ -- in particular it grows faster than any tower of exponentials of fixed height. We ask whether any one-relator group has a larger conjugator length function than the Baumslag-Gersten group. Along the way, we also show that the conjugator length function of the $m$-th iterated Baumslag-Solitar groups is equivalent to the $(m-1)$-times iterated exponential function.
\end{abstract}
\section{Introduction}

% All of our results in this paper will be stated in terms of the following equivalence relation. For functions $f,g:\Naturals\to \Naturals$, we say $f(n)\preceq g(n)$ if there exists some constant $C>0$ such that $f(n)\leq Cg(Cn+C)+C$ for all $n$, and we say $f(n)\simeq g(n)$ if $f(n)\preceq g(n)$ and $g(n)\preceq f(n)$. Additionally, we will use the the iterated-exponential function $E(m,n)$, defined inductively: $E(0,n)=n$, $E(m+1,n)=2^{E(m,n)}$.

Let $G$ be the one-relator group $\langle a,t\mid tat^{-1}a(tat^{-1})^{-1}=~a^{2}\rangle$, referred to as the Baumslag-Gersten group. The word problem of $G$ and related questions have received considerable attention in the literature. For example, the Dehn function of $G$ is equal (up to a natural equivalence $\simeq$; see Definition \ref{simeqdef}) to an exponential power-tower of height $\lfloor \log_2 n\rfloor$ \cite{platonov2004isoparametric}; in particular, the Dehn function of $G$ grows faster than any tower of exponentials of fixed height (see \cite{Gersten1992} and \cite{RileyDehn}). More recently, the word problem was shown to be solvable in time $O(n^7)$ using the technology of ``power-circuits" to efficiently represent and compute with certain large numbers \cite{MYASNIKOV2011324,miasnikov2012power}. The conjugacy problem in $G$ is also decidable, however the current best estimate of its complexity is not elementary-recursive \cite{beese2012konjugationsproblem}. Restricting to the special case of $u,v\in X=G\smallsetminus\langle a,tat^{-1}\rangle_G$, the authors of \cite{DMW} give an $O(n^4)$-time algorithm deciding whether $u$ and $v$ are conjugate, and that this set $X$ is ``generic" in a natural way, so the conjugacy problem for $G$ is generically decidable in polynomial time. They further conjecture that there is no algorithm whose worst-case or average running time is polynomial. Most recently, Mattes and Wei{\ss} showed  that, for fixed $g\in G$, deciding whether a given element is conjugate to $g$ is in $\mathsf{uTC}^1$ \cite{MattesWeiss}.

 Now, fix a finite presentation $\langle S\mid R\rangle$ of a finitely presented group $H$, and let  $\abs{\cdot}_H$ be the resulting word metric on $H$. Given two words $u$ and $v$ on $S$ representing conjugate elements of $H$, we define $c(u,v)$ to be the smallest value of $\abs{\gamma}_H$ over all $\gamma\in H$ conjugating $u$ to $v$. The \textit{conjugator length function} $\CL_H(n)$ of $H$ is the maximum value of $c(u,v)$ over all conjugate pairs $(u,v)$ such that $\abs{u}_H+\abs{v}_H\leq n$.  Although $\CL_H(n)$ depends on the presentation chosen for $H$, it can be shown that its $\simeq$-equivalence-class does not.

Conjugator length functions have lately gained attention as a way of quantifying the difficulty of the conjugacy problem in finitely presented groups. This is in close analogy with the relationship between the word problem and the Dehn function. For example, modulo the complexity of the word problem, the conjugator length function can be used to find an upper bound for the computational complexity of the conjugacy problem in a group. We refer the reader to \cite{BRS} for an up-to-date survey on this and other applications.

 We continue the study of the conjugacy problem of $G$ by computing close upper and lower bounds for $\CL_G(n)$. As preparation, we also compute exactly the conjugator length function of the \textit{$m$-th iterated Baumslag-Solitar group} $G_m=\langle s_0,s_1,\ldots, s_m\mid s_is_{i-1}s_i^{-1}=s_{i-1}^2\ \forall i=1,\ldots, m\rangle$ for all $m$. To state our results fully, we require the following definitions:
 
 \begin{defn}\label{simeqdef}Let $f,g:\Naturals\to\Naturals$. We say $f(n)\preceq g(n)$ if there exists some $C>0$ such that $f(n)\leq Cg(Cn+C)+Cn+C$ for all $n\in \Naturals$, and $f(n)\simeq g(n)$ if $f(n)\preceq g(n)$ and $g(n)\preceq f(n)$.
 \end{defn}
 \begin{defn}
     For $n\in \Naturals$, define the iterated exponential function $E(m,n)$ inductively as follows: $E(0,n)=n$, $E(m,n)=2^{E(m-1,n)}$.
 \end{defn}

 Our first result is the following:
\begin{thm}\label{FirstTheorem}
    We have $\CL_{G_m}(n)\simeq E(m-1,n)$.
\end{thm}

 The upper bound of this statement is shown in Proposition \ref{GmUpperBound}, which is proved by first conjugating $u$ and $v$ into the subgroup $\langle s_0,\ldots, s_r\rangle_{G_m}$, $r$ being the smallest number such that $u$ and $v$ are conjugate to an element of $\langle s_0,\ldots, s_r\rangle_{G_m}$ (Lemma \ref{conjugatingToRank}). The proof of Proposition \ref{GmUpperBound} then proceeds by analysis of annular diagrams over $G_m$ (defined in Section \ref{AnnDiag}) using facts about the distortion of the subgroups $\langle s_i\rangle_{G_m}$ and conjugator length in the subgroups $\langle s_i,s_{i-1}\rangle_{G_m}$ (Lemmas \ref{powerOfGeneratorLength} and \ref{conjugateEntirelyInTwoConsecutiveGenerators} respectively).

 The lower bound is shown in in Proposition \ref{IteratedLowerBound}, where we follow \cite{MYASNIKOV2011324, miasnikov2012power} in representing each $k\in \ZZ$ as sums of the form $\sum_j n_j2^{m_j}$. However, while these papers use representations of this sort to efficiently perform arithmetic operations on large numbers, we use them to find lower bounds for $\abs{s_0^k}_{G_m}$ for specific values of $k$ (Corollary \ref{lengthOfCompactPower}). We then  show $|s_1^2|_{G_m}+|s_0^{4E(m,n)-1}s_1^2|_{G_m}\leq Cn$ for some constant $C>0$ (Lemma \ref{lengthOfPowersOf2OfGenerators}). The proof proper of Proposition \ref{IteratedLowerBound} begins by showing that any minimal length conjugator taking $s_1^2$ to $s_0^{4E(m,n)-1}s_1^2$ must be equal (in $G_m$) to $s_1^\alpha s_0^\beta$, where $\alpha\leq 0$ and $\beta=-2^{-\alpha}\cdot (4E(m,n)-1)/3$ (the factor of $-\frac{1}{3}$ arises from the fact that $s_0^{\beta}s_1^2s_0^{-\beta}=s_0^{-3\beta}s_1^2$). Using the triangle inequality and a basic fact about the structure of $G_m$ (Corollary \ref{retractShrinksWords}), we obtain $|s_1^\alpha s_0^\beta|_{G_m}\geq  |s_0^\beta|_{G_m}/2$, and complete the proof of the lower bound by applying Corollary \ref{lengthOfCompactPower} in the case where $k=\beta$. The length of the binary representation of $(E(m,n)-1)/3$ is the key element in applying this Corollary.

The second result of our paper is:
\begin{thm}\label{SecondTheorem}
    We have $ \CL_G(n)\simeq  E(\lfloor \log_2 n\rfloor,1)$.
\end{thm}

For an intermediate upper bound of $nE(\lfloor \log_2 n\rfloor,1)$, obtained in Proposition \ref{upperBoundGersten}, we adopt a new presentation (Lemma \ref{GerstenSubgroupIterated} and Remark \ref{newGerstenPres}) for $G$: $$G=\langle s_0,s_1,t\mid s_1s_0s_1^{-1}=s_0^2,\  ts_0t^{-1}=s_1\rangle.$$ Then, we recall Platonov's computation of the distortion functions of $\langle s_0\rangle_{G}$,$\langle s_1\rangle_{G}$ in $G$ \cite{platonov2004isoparametric} (Lemma \ref{GerstenDistortionLemma} and Corollary \ref{GerstenDistortionCor}), and conjugate $u$ and $v$ to specific cyclically Britton-reduced conjugacy-class representatives (Lemma \ref{cycConjGersten}). If $u,v\in X$, the desired upper bound follows by the length of the conjugator returned by the algorithm in \cite[Theorem 3]{DMW}. For all other elements of $G$, our upper bound is established via a diagrammatic argument similar in spirit to that for Proposition~\ref{GmUpperBound}, except that more care is needed to account for the generator $t$. Corollary \ref{finalGerstenUpper} then obtains the desired final upper bound, showing $nE(\lfloor \log_2 n\rfloor,1)\preceq E(\lfloor \log_2 n\rfloor,1)$.

For the lower bound, shown in Proposition \ref{lowerBoundGersten}, we reduce to consideration of conjugacy in the subgroup $$\langle t^{-m}s_0t^m,\ldots,t^{-1}s_0t,s_0,ts_0t^{-1},\ldots, t^ms_0t^{-m}\rangle_{G}$$ for some specific $m$. We show (Lemma \ref{GerstenSubgroupIterated}) that this subgroup is isomorphic to $G_{2m}$, which lets us appeal to the lower bound of Theorem \ref{FirstTheorem}, and in particular Corollary \ref{lengthOfCompactPower}, to complete the proof.

In contrast to the word problem for one-relator groups \cite{Magnus_Karrass_Solitar}, the decidability of their conjugacy problem is still an open question. This paper quantifies the conjugacy problem in one well-known case and shows that, in a sense (see \cite{BRS}), it is of similar difficulty to the word problem of that case. Inspired by Gersten's conjecture in \cite{Gersten1992} that the Dehn function of $G$ is the largest among all one-relator groups, we pose the following question:
\begin{Question}
For all one-relator groups $H$, is $\CL_H(n)\preceq \CL_G(n)$?
\end{Question}

A positive answer to this question would, by standard techniques, give a solution to the conjugacy problem for one-relator groups, while a counterexample may better illuminate some of the intrinsic difficulties to be overcome in finding such a solution.

\subsection{Acknowledgments} The author thanks Tim Riley for his very helpful advice and comments. This work was generously supported by NSF Grant DGE–2139899.

\subsection{AI Usage} The Claude LLMs Opus and Fable were used to aid in proofreading and to prove a stronger version of Lemma 6.3 than existed in a previous version of this paper, which in turn strengthened Theorem 1.4. The author accepts full responsibility for the paper's contents.

\section{Preliminaries} 
\subsection{Conventions and Notation}
\begin{defn}
    A \textit{word} $w$ on a set $X$ is a string of elements of $X$ and their formal inverses, denoted $x^{-1}$ for each $x\in X$. We say $w$ is \textit{reduced} if it contains no subwords of the form $xx^{-1}$ or $x^{-1}x$. A \textit{cyclic conjugate} of $w$ is a word $w'$ which is a cyclic permutation of the letters of $w$, and $w$ is \textit{cyclically reduced} if every cyclic conjugate is reduced, or equivalently, if $w^2$ is reduced. If $X$ is the generating set of some group, we will abuse notation and simply write $w$ for the element represented by $w$.
\end{defn}

     \begin{defn}
         Let $H$ be a group with a fixed finite presentation $\langle S\mid R\rangle$. For all $h\in H$, $|h|_H$ is the minimal length of word on $S$ representing $h$ in $H$. We call $|\cdot|_H$ the \textit{word-metric} on $H$. All of the groups in this paper will be finitely presented, and the presentation being used will be clearly specified. We will refer to the generating sets of the given presentations as \textit{chosen generators}. Finally, we state at the outset that no two defining relations of any presentation will be cyclic permutations of each other.
     \end{defn}

\subsection{Conjugator Length and Annular Diagrams}\label{AnnDiag}

We assume the reader is familiar with the theory of van Kampen diagrams over finitely presented groups as presented in \cite{lyndon2001combinatorial}, and focus solely on annular diagrams. For this subsection, let $H=\langle S\mid R\rangle$ be a finitely presented group.

\begin{defn}\label{AnnularDiagram}Let $u$ and $v$ be words on $S$ representing conjugate elements of $H$. An \textit{annular diagram} $\Delta$ for $u$ and $v$ is a planar, oriented $2$-cell complex with edges labeled by $S$, along with points $P_1,P_2$ on the outer and inner boundary components respectively, such that
\begin{enumerate}
    \item \label{AD1} $\Delta$ is an annulus,
    \item \label{AD2} starting from $P_1$ and $P_2$, the words along the outer and inner boundary components (oriented clockwise) are equal to $u$ and $v$ respectively, and
    \item \label{AD3} starting from some point along the boundary $\del C$ of a 2-cell $C$ in $\Delta$, the word along $\del C$ is an element of $R^{\pm1}$.
\end{enumerate}
In such an annular diagram, the length $D(\Delta)$ of the shortest path between $P_1$ and $P_2$ is the \textit{diameter}. Also, a 2-cell $C$ in $\Delta$ is said to \textit{correspond} to any $r\in R$ such that $r^{\pm1}$ is (a cyclic conjugate of) the word along $\del C$. By our assumption that $R$ contains no elements which are cyclic permutations of each other, each 2-cell corresponds to a unique $r\in R$.
\end{defn} 

For the following fundamental result, see Lemmas V.5.1 and V.5.2 of \cite{lyndon2001combinatorial} and their proofs:
\begin{lemma}
     Two words $u$ and $v$ on $S$ represent conjugate elements of $H$ if and only if there exists an annular diagram for $u$ and $v$. Moreover, $c(u,v)=\min\{D(\Delta)\}$, where this minimum is taken over all annular diagrams $\Delta$ for $u$ and $v$.
\end{lemma}
Now, we sketch a proof of a technical lemma we will use below.
\begin{lemma}\label{exciseTwoCorridors}
Let $\gamma_1$ and $\gamma_2$ be two non-crossing simple loops in an annular diagram $\Delta$ such that $\gamma_1,\gamma_2$ have the same non-zero winding number around the inner boundary component of $\Delta$, and let $\Delta_i$ be the annular subdiagram bounded by $\gamma_i$ and the inner boundary. Both $\Delta_1$ and $\Delta_2$ are homeomorphic to annuli, with either $\Delta_1\subseteq \Delta_2$ or $\Delta_2\subseteq \Delta_1$. Moreover, if  the words along $\gamma_1$ and $\gamma_2$ are the same, then the subdiagram $\Delta_2\smallsetminus \Delta_1$ can be excised to obtain a diagram with smaller or equal diameter.
    
\end{lemma}

\begin{proof}[Proof Sketch:] The first conclusion follows immediately from the assumptions on $\gamma_1,\gamma_2$. For the second, proceeding naively, one can simply delete all of $\Delta_2\smallsetminus \Delta_1$ and identify the paths $\gamma_1$ and $\gamma_2$ together. There exist cases, however, where this operation may violate planarity. As in the proof van Kampen's Theorem, this can be handled by the operations of ``sewing" and ``detachment": identifying each pair of edges of the paths sequentially, and removing any 2-cells which violate planarity (see pg. 150-151 of \cite{lyndon2001combinatorial} for details). 
    
\end{proof}

\subsection{Corridors}
\begin{defn}\label{defCorrs}
Except for the single defining relation in the presentation $G=\langle a,t\mid (tat^{-1})a(tat^{-1})^{-1}=a^2\rangle$, all of the defining relations considered in this paper will be written in the form $xwx^{-1}\tilde{w}^{-1}$ for some $x$ not appearing in $w$ or $\tilde{w}$. In a (van Kampen or annular) diagram $\Delta$, a maximal sequence $C_1,C_2,\ldots, C_k$ of 2-cells such that $C_i$ and $C_{i+1}$ share an edge labeled by $x$ is called an $x$\textit{-corridor}. A single $x$-edge on the boundary of a diagram is called a degenerate \textit{$x$-corridor} if it is not contained in any 2-cells. The \textit{positive side} of cell corresponding to $xwx^{-1}\tilde{w}^{-1}$ is the side labeled by $w$, and the \textit{negative side} is that labeled by $\tilde{w}$ -- the positive and negative sides of corridors are defined analogously.
\end{defn}

\begin{defn}\label{defReduced}
    Suppose $\Delta$ has two 2-cells $C_1$ and $C_2$ which share at least one edge with an endpoint $P$, such that the word along $\del C_1$ read clockwise from $P$ is the inverse of the word along $\del C_2$ read clockwise from $P$. Then $C_1$ and $C_2$ are called \textit{canceling 2-cells}. If we excise $C_1$ and $C_2$, then apply the same ``sewing" and ``detachment" argument as Lemma \ref{exciseTwoCorridors}, we obtain a diagram $\Delta'$ with the same word(s) along the boundary, up to equality in $G$, and strictly fewer canceling pairs. A diagram is called \textit{reduced} if it contains no canceling 2-cells. 
\end{defn}

The following lemma follows immediately from Definitions \ref{defCorrs} and \ref{defReduced}.
\begin{lemma}\label{reducedAlongCorrs}
    Let $H=\langle S\mid R\rangle$ be a finitely presented group and let $x\in S$. Suppose there is only one relation in $R$ of the form $xwx^{-1}\tilde{w}^{-1}$, and suppose $w,\tilde{w}$ are cyclically reduced.
    % , and that either the lengths of $w$ and $\tilde{w}$ both equal one or both $w$ and $\tilde{w}$ are powers of the same generator. 
    Then in a reduced annular diagram $\Delta$ over $H$, the words along the sides of any $x$-corridor in $\Delta$ are reduced.
\end{lemma}

% \blue{TO CONAN: Include discussion here. Be sure to mention how to make all corridors freely reduced without increasing the lengths of any paths, and the taxonomy of corridors. Also, explain how two concentric corridors with the same words along the sides can be excised. Also describe relationship to conjugator length.}

\begin{defn}There are four mutually exclusive possibilities for an $x$-corridor in an annular diagram $\Delta$:

\begin{enumerate}
    \item it contains an $x$-edge on each boundary component of $\Delta$,
    \item either it contains two $x$-edges on the same boundary component of $\Delta$, or is degenerate,
    \item it forms a loop which bounds a contractible subdiagram of $\Delta$, and
    \item it forms a loop which bounds a non-contractible subdiagram of $\Delta$.
\end{enumerate}
The $x$-corridors satisfying each of these conditions are called \textit{radial $x$-corridors, $x$-arches, contractible $x$-rings,} and \textit{non-contractible $x$-rings} respectively. If $Q$ is an $x$-corridor satisfying one of the latter two cases, then $Q$ is itself an annulus; the \textit{outer word} and and \textit{inner word along $Q$} are defined accordingly. 
\end{defn}

\begin{defn}
    We say two radial $x$-corridors in $\Delta$ are \textit{consecutive} if they bound (along with the boundary components of $\Delta$) a sub-diagram containing no radial $x$-corridors; if there is only one such corridor in $\Delta$, we adopt the convention that it is consecutive to itself. 
\end{defn}

\subsection{Iterated Baumslag-Solitar Groups} 
Recall the presentation $G_m=\langle s_0,\ldots, s_m\mid s_is_{i-1}s_i^{-1}=s_{i-1}^2\ \forall i=1,\ldots, m\rangle$ of the $m$-th iterated Baumslag-Solitar group.

\begin{lemma}\label{nestedBS} Let $i_m:G_m\to G_{m+1}$ be the map given by $s_i\mapsto s_i$ for $i=0,\ldots,m$. Then $i$ is injective.
\end{lemma}
\begin{proof} Since $G_1=\BS(1,2)$ is torsion free, $s_1$ has infinite order in $G_1$. Assuming by induction that $s_m$ has infinite order in $G_m$, taking the HNN-extension with stable letter $s_{m+1}$ along the homomorphism $\langle s_{m}\rangle\to\langle s_{m}^2\rangle$ gives the group $G_{m+1}$ with the inclusion map $G_m\hookrightarrow G_{m+1}$ equal to $i_m$, and with $s_{m+1}$ having infinite order in $G_{m+1}$.
\end{proof}
\begin{lemma}\label{retractionToLowerGroups}
    The homomorphism $\pi_m:G_m\to G_{m-1}$ given by $s_0\mapsto 1$ and $s_i\mapsto s_{i-1}$ for $i\geq 1$ is a retraction.
\end{lemma}
\begin{proof}
    The homomorphism $\pi_m':G_{m-1}\to G_m$ defined by $s_i\mapsto s_{i+1}$ is a right-inverse to $
    \pi_m$.
\end{proof}
\begin{cor}\label{noLowerGeneratorsInMinimumLengthWord}
    If $g\in \langle s_i,\ldots,s_m\rangle_{G_m}$, then there exists a word $w$ on $\{s_i,\ldots,s_m\}$ of length $|g|_{G_m}$ such that $w=g$ in $G_m$.
\end{cor}
\begin{proof}
    Let $w_0$ be any word of length $|g|_{G_m}$ on $\{s_0,\ldots,s_m\}$ representing $g$, and let $w$ be the word resulting from deleting $s_j^{\pm1}$ from $w_0$ for $j<i$. By Lemma \ref{retractionToLowerGroups}, $$w=(\pi'_m\circ \pi_{m-1}'\circ\cdots\circ \pi'_{m-i+1}\circ\pi_{m-i+1}\circ\pi_{m-i+2}\circ\cdots\circ \pi_m)(w_0)=g.$$ Moreover, $w$ is a word on $\{s_i,\ldots, s_m\}$ such that $|w|\leq |w_0|$.
\end{proof}
\begin{cor}\label{retractShrinksWords}
    For $\alpha,\beta\in \ZZ$, $|s_1^\alpha s_0^\beta|_{G_m}\geq |s_1^\alpha |_{G_{m}}$
\end{cor}
\begin{proof}
    Let $w_0$ be any representative word for $s_1^\alpha s_0^\beta$ in $G_m$, and $w$ the word resulting from deleting all $s_0$-letters from $w_0$. Then $w=\pi_m'(\pi_m(w_0))=s_1^\alpha$ in $G_m$ and has length at most $|w_0|$.
\end{proof}

\begin{lemma}\label{logarithmicShrinking}
    For all $i=0,\ldots,m-1,$ $|s_i^k|_{G_m}\leq C\log_2|k|$ for some constant $C>0$.
\end{lemma}
\begin{proof}
    The claim is straightforward for $G_1$. For $m>1$, let $$H=\langle s_i,s_{i+1}\rangle_{G_m}=(\pi_{i-1}'\circ \pi_{i-2}'\circ\cdots\circ \pi_2')(G_1).$$ This is isomorphic to $G_1$, and $|s_i^k|_{G_m}\leq |s_i^k|_{H}$ because the chosen generating set of $H\leq G_m$ is a subset of the chosen generating set of $G_m$, whence the claim follows. 
\end{proof}

\begin{lemma}\label{largestIndexInDiagram}
    Let $m', m''\leq m$, and let $u $ be a word on $\{s_0,\ldots, s_{m'}\}$ and $v$ a word on $\{s_0,\ldots,s_{m''}\}$ containing an $s_{m'}$- and $s_{m''}$-letter respectively, such that $u$ and $v$ represent conjugate elements of $G_m$. Let $\Delta$ be a reduced annular diagram for $u$ and $v$, and let $k$ be the largest number such that there is an edge in $\Delta$ labeled $s_k$. Then all $s_k$-edges in $\Delta$ appear only in $s_k$-corridors. Moreover, if $\Delta$ has no radial $s_k$-corridors or $s_k$-arches, then $m'=m''=k-1$, otherwise, $k=\max\{m', m''\}$.
\end{lemma}
\begin{proof}
   That all $s_k$-edges must appear in $s_k$-corridors holds because the only relations in the presentation for $G_m$ containing an $s_k$-letter yet no $s_{k+1}$ letters is $s_ks_{k-1}s_k^{-1}=s_{k-1}^2$. Moreover, $k\geq \max\{m',m''\}$ by definition. Note that $\Delta$ has an $s_k$-corridor which is radial or an arch if and only if $u$ or $v$ has an $s_k$-letter, which occurs if and only if $k\leq \max\{m',m''\}$, hence $k=\max\{m',m''\}$ in this case. 
   
   Suppose instead that there are no are radial $s_k$-corridors or $s_k$-arches in $\Delta$. Then $u$ and $v$ contain no $s_k$-letters, so $\max\{m',m''\}<k$. Every $s_k$-edge is part of either a contractible or non-contractible $s_k$-ring. If there exists a contractible $s_k$-ring, Lemma \ref{reducedAlongCorrs} implies the word along its outer-boundary is a freely reduced word on $s_{k-1}$ representing the identity in $G_m$. In the proof of Lemma \ref{nestedBS} we observed that $s_i$ has infinite order for all $i$, hence no such $s_k$-corridor can exist. Thus every $s_k$-edge must be contained in a non-contractible $s_k$-ring. Let $Q$ be the $s_k$-ring closest to the boundary component labeled by $u$, and let $\Delta'$ be the diagram contained between $Q$ and that boundary component. Since $\Delta'$ contains no $s$-edges with index greater than or equal to $k$ we see $\Delta'$ is an annular diagram over $G_{k-1}$ and $\Omega=\Delta'\cup Q$ is an annular diagram over $G_{k}$. By applying Lemma \ref{reducedAlongCorrs} to $\Omega$, then, the word $w$ along the side of $Q$ adjacent to $\Delta'$ is a freely reduced power of $s_{k-1}$. Any $s_{k-1}$ arch in $\Delta'$ runs between $s_{k-1}$-edges of opposite sign, hence every $s_{k-1}$-edge in $w$ is contained in a radial $s_{k-1}$-corridor in $\Delta'$. The set of $s_k$-edges, and thus the set of 2-cells, of $Q$ is non-empty, so this implies $k-1\leq m'$. By a similar argument, $k-1\leq m''$, and we are done.\end{proof}

\subsection{The Baumslag-Gersten group}
Recall the presentation $G=\langle a,t\mid tat^{-1}a(tat^{-1})^{-1}=a^2\rangle$ of the Baumslag-Gersten group.

\begin{lemma}\label{GerstenSubgroupIterated}
    The subgroup $\langle a,tat^{-1},t^2at^{-2},\ldots, t^mat^{-m}\rangle_{G}$ of $G$ is isomorphic to $G_m$ via a map which restricts to a bijection between the chosen generating sets. 
\end{lemma}
\begin{proof}
    By Lemmas \ref{nestedBS} and \ref{retractionToLowerGroups}, the subgroups of $G_m$ generated by $\{s_1,\ldots, s_m\}$ and $\{s_0,\ldots, s_{m-1}\}$ are isomorphic to each other by the map $i_{m}\circ \pi_{m-1}$. Let $G'$ be the HNN extension along this isomorphism with stable letter $t$, so $$G'=\langle s_0,s_1,\ldots, s_m,t\mid s_js_{j-1}s_j^{-1}=s_{j-1}^2,\ ts_{j-1}t^{-1}=s_{j}\ \forall j=1,\ldots, m\rangle .$$  For $j>1$, the relation $s_js_{j-1}s_j^{-1}=s_{j-1}^2$ follows from the relations $s_1s_0s_1^{-1}=s_0^2$ and $s_j=t^js_0t^{-j}$. We may thus omit these relations, as well as the generators $s_j$ for $j>1$, giving the presentation as \begin{equation}\tag{$\ast$}\label{BGPres}G'=\langle s_0,s_1,t\mid s_1s_{0}s_1^{-1}=s_0^2,\ ts_0t^{-1}=s_{1}\rangle .\end{equation} The groups $G'$ and $G$ are isomorphic via the homomorphism determined by $s_0\mapsto a,s_1\mapsto tat^{-1},t\mapsto t$. This map takes $s_i$ to $t^iat^{-i}$, so we are done.
\end{proof}

\begin{rem}\label{newGerstenPres}
    In light of the above proof, we will use the presentation \eqref{BGPres} for $G$ throughout the rest of this paper. The length function $|\cdot |_G$ will always be with respect to this presentation, and  we will make no distinction in our notation between an element of $\langle s_0,\ldots, s_m\rangle_{G}$ and an element of $G_m$. Lastly, we generalize the notation $s_j$ to all integers by defining, for $i>0$, $s_{-i}=t^{-i}s_0t^{i}$; note the relation $s_{j}s_{j-1}s_{j}^{-1}=t^js_1s_0s^1t^{-j}=t^js_0^2t^{-j}=s_{j-1}^2$ still holds for all $j\in \ZZ$.
\end{rem}
% \begin{lemma}\label{IteratedToGerstenGeneratorWords}
%      If $g\in\langle s_0,\ldots, s_m\rangle_{G}$, then $|g|_G\leq 3m|g|_{G_m}$.
% \end{lemma}
% \begin{proof}
%     For $\ell=|g|_{G_m}$, we may represent $g$ as a product $s_{i_1}^{\pm1}s_{i_2}^{\pm1}\cdots s_{i_\ell}^{\pm1}$, where $i_j\leq m$ for all $j$. Replacing each $s_{i_j}$ with $t^{i_j}s_0t^{-i_j}$ increases the length by a factor of at most $2m+1\leq 3m$, and gives a word representing $g$ on the generators of $G$ -- the claim follows.  
% \end{proof}

 For the next lemma, let $H_m=\langle s_{-m},s_{-m+1},\ldots,s_0,s_1,\ldots, s_m\rangle_{G}$. After conjugating by $t^m$ and appealing to Lemma \ref{GerstenSubgroupIterated}, we see that this group is isomorphic to $\langle s_{0},s_1,\ldots,s_{2m}\rangle_{G}\simeq G_{2m}$ via a bijection on the chosen generating sets, so (committing a minor abuse of notation) $|x|_{H_m}=|t^mxt^{-m}|_{G_{2m}}$.
\begin{lemma}\label{lowerWordBoundGersten}
     Let $g\in \langle \ldots, s_{-2}, s_{-1},s_0,s_1,\ldots\rangle$, and let $M=|g|_G$. Then $g\in H_{M+1}$ and $M\geq |g|_{H_{M+1}}$.
\end{lemma}
\begin{proof}
     Let $w$ be a minimal length word on $\{s_0,s_1,t\}$ representing $g$, so $|w|\leq M$, and let $\tau:G\to \ZZ$ be the homomorphism killing $s_0,s_1$ and sending $t$ to $1$. Then the kernel of $\tau$ contains $s_i=t^is_0t^{-i}$ for all $i\in \ZZ$, so $\tau(g)=0$ and the $t$-exponent sum of $w$ is therefore $0$. Let $w'$ be the word given by replacing $s_1^{\pm1}$ with $ts_0^{\pm1}t^{-1}$, noting the $t$-exponent sum of $w'$ is still 0. For each instance of $s_0^{\pm1}$, the $t$-exponent sum $\sigma$ of the prefix of $w'$ leading to= that $s_0$-letter has equal magnitude and opposite sign as the $t$-exponent sum of the suffix following it. Noting $|\sigma|\leq M$, we may replace this $s_0$ with the word $t^{-\sigma}s_\sigma^{\pm1} t^\sigma$, which still equals $s_0$. By construction the prefix before the new $s_\sigma$ now has $t$-exponent sum $0$, but the overall $t$-exponent sum of $w'$ and the exponent sum of the prefix leading to every other $s_0$-letter in $w'$ are left unchanged. Let $w''$ be the result after replacing each $s_0$ instance from left to right, freely reducing each time. Since the prefix before each $s_k$-letter, as well as the entirety of $w''$, have $t$-exponent sum $0$, and since $w''$ is freely reduced, $w''$ contains no $t$-letters. Thus $w''$ is a word on $S=\{s_{-M-1},s_{-M},\ldots,s_{-1},s_0,s_1,\ldots, s_{M+1}\}$, the chosen generating set of $H_{M+1}$, with $|w''|\leq |w|$, so we are done.
\end{proof}

\subsection{Rank and Pinches}
\begin{defn}\label{rankDef}
    For $g\in G_m$, we define the \textit{rank} $\rk(g)$ as the smallest $r\leq m$ such that $g$ is conjugate to an element of $\langle s_0,\ldots, s_r\rangle_{G_m}=G_r$.
\end{defn}
\begin{defn}\label{pinchAndReducedDef}
    Given a word $w$ on $\{s_0,\ldots,s_m\}$, we say a subword $P$ is a \textit{pinch} if it is of the form $P=s_i^{\delta}w's_i^{-\delta}$, where $w'\in \langle s_{i-1}\rangle$ if $\delta=1$ and $w'\in \langle s_{i-1}^2\rangle$ if $\delta=-1$, and where $i>\rk(w)$. We say every subword of $w'$ is \textit{involved} in the pinch $P$, and call the first and last letters of $P$ the \textit{endpoints} of $P$. When specificity is required, we will also call $P$ an $i$\textit{-pinch}. Involvement induces a partial order on $i$-pinches of cyclic conjugates of $w$. 
    %Note that $w'$ may be trivial in an $i$-pinch.
\end{defn}

\begin{Example}
    In $G_2$, consider the word $s_2s_1s_2^{-1}s_0s_1^{-2}$. The only pinch is the 2-pinch $s_2s_1s_2^{-1}$. If we replace it with $s^2_1$ to obtain $s_1^2s_0s_1^{-2}$, there now exist two 1-pinches, $s_1^2s_0s_1^{-2}$ and $s_1s_0s_1^{-1}$.
\end{Example}

\begin{defn}\label{tPinch}
    Let $w$ be a word on $\{s_0,s_1,t\}$. A \textit{$t$-pinch} is a subword of the form $t^{\delta}w't^{-\delta}$ with $\delta=\pm1$, where $w'\in \langle s_0\rangle$ if $\delta=1$ or $w'\in\langle s_1\rangle$ if $\delta=-1$. The terms \textit{involvement} and \textit{endpoints} are defined analogously.  A $t$-letter is an instance of one of $t$ or $t^{-1}$, and two $t$-letters $t_1$ and $t_2$ are \textit{$t$-partners} if they are the endpoints of a $t$-pinch in some cyclic conjugate of $w$. As with $i$-pinches in $G_m$, involvement is a partial order on the set of $t$-pinches in cyclic conjugates of $w$. If no cyclic conjugate of $w$ has any $t$-pinches, then we say it is \textit{cyclically Britton-reduced} following \cite{DMW}.
\end{defn}

We conclude this section with two useful results concerning pinches in $G_m$.

\begin{lemma} \label{higherGeneratorsImpliesPinch}
    Suppose $w$ is a word on $\{s_0,\ldots, s_j\}$  such that $\rk(w)<j$ and $w$ contains at least one $s_j$-letter. Then there exists a set $\mathcal{P}$ of $j$-pinches of cyclic conjugates of $w$ such that every $s_j$-letter is an endpoint of some $j$-pinch in $\mathcal{P}$ and, for all $P_1,P_2\in \mathcal{P}$, either they have no letters in common or one is involved in the other.
\end{lemma}
\begin{proof}
 By the definition of $\rk(w)$, we see $w$ is conjugate to a word $w'$ on $\{s_0,\ldots, s_{\rk(w)}\}$. By Lemma \ref{largestIndexInDiagram}, all $s_j$-letters in $w$ are part of $s_j$-arches in any annular diagram witnessing the conjugacy of $w$ and $w'$. Every $s_j$-arch corresponds to a $j$-pinch in some cyclic conjugate of $w$, and no $s_j$-arches can cross. Therefore, the collection of such arches gives the desired $\mathcal{P}$.
\end{proof}
Replacing annular diagrams with van Kampen diagrams in this argument, we also obtain:
\begin{lemma}\label{higherGeneratorInPinchVanKampen}
    Let $w$ be a word on $\{s_0,\ldots, s_{j}\}$. If $w\in\langle s_0,\ldots, s_{j-1}\rangle$ there exists a set $\mathcal{P}$ of $j$-pinches of $w$ such that every $s_j$-letter is an endpoint of some $j$-pinch in $\mathcal{P}$ and, for all $P_1,P_2\in \mathcal{P}$, either they have no letters in common or one is involved in the other.
\end{lemma}
Note that the elements of $\mathcal{P}$ here are pinches in $w$, not cyclic conjugates of $w$.
\section{Upper Bound For $G_m$}

    \begin{lemma}\label{powerOfGeneratorLength}
        Let $w$ be a word on $\{s_0,s_1,\ldots, s_{i+1}\}$ of length $n$ with $p$ many $s_{i+1}$-letters and representing an element of $\langle s_i\rangle_{G_m}$. Then $w=s_i^\ell$ for $|\ell|\leq (n-p)2^p$. In particular, $|\ell|\leq 2^n$.
    \end{lemma}
    \begin{proof}
        Similarly to Corollary \ref{noLowerGeneratorsInMinimumLengthWord}, let $w'$ be the result of applying $\pi'_m\circ \pi_{m-1}'\circ\cdots\circ \pi'_{m-i+1}\circ\pi_{m-i+1}\circ\pi_{m-i+2}\circ\cdots\circ \pi_m$ to every letter of $w$, so $$w'=(\pi'_m\circ \pi_{m-1}'\circ\cdots\circ \pi'_{m-i+1}\circ\pi_{m-i+1}\circ\pi_{m-i+2}\circ\cdots\circ \pi_m)(w).$$ Note that $w'$ is a word on $\{s_i,s_{i+1}\}$ of length at most $|w|$, and that $w=w'$ in $G_m$ since $w$ represents an element of $\langle s_i\rangle_{G_m}$, which is fixed by the above composite mapping. The claim follows by a straightforward computation on $w'$, viewed as an element of $\langle s_{i},s_{i+1}\rangle_{G_m}=\BS(1,2)$.
    \end{proof}
\begin{lemma}\label{conjugatingToRank}
    Let $w$ be a word on $\{s_0,\ldots, s_m\}$ of rank $r$ and length $n$, and let $r\leq i\leq m$. Then there exists a word $w_i$ on $\{s_0,\ldots, s_i\}$ of length at most $E(m-i,n)$ which is conjugate in $G_m$ to $w$ via a word $\gamma_i$ on $\{s_0,\ldots, s_m\}$ of length at most $(m-i)E(m-i-1,n)$. \end{lemma}     
     \begin{proof}
         We proceed by downward induction on $i$. The base case is trivial. For the inductive step, let $w_{i+1}$ and $\gamma_{i+1}$ be given. If $w_{i+1}$ has no $s_{i+1}$-letters, we are done. Otherwise, by Lemma \ref{higherGeneratorsImpliesPinch} there exists a set $\mathcal{P}$ of $(i+1)$-pinches in $w_{i+1}$ such that every $s_{i+1}$-letter is involved in or an endpoint of some $p\in \mathcal{P}$, and, for all $P_1,P_2\in \mathcal{P}$, one is involved in the other or they have no letters in common. By the latter property of $\mathcal{P}$, there exists at least one maximal pinch $P\in \mathcal{P}$ of some cyclic conjugate of $w_{i+1}$ under the relation of involvement. Let $\gamma$ be the prefix of $w_{i+1}$ ending at the leftmost endpoint of $P$ if $P$ is a pinch of $w_{i+1}$, or at the rightmost endpoint if $P$ is not. Every element of $\mathcal{P}$ is a pinch in the cyclic conjugate $w_i'=\gamma^{-1}w_{i+1}\gamma$ of $w_{i+1}$.
         
         We may write $w_i'=q_0P_1q_1P_2\cdots q_{k-1}P_k q_{k}$, where each $P_i$ is a maximal element of $\mathcal{P}$ and each subword $q_j$ is disjoint from every element of $\mathcal{P}$. Note that $$|q_0|+\sum_{j=1}^k\left(|q_j|+|P_j|\right)= |w_i'|=|w_{i+1}|\leq E(m-i-1,n),$$ and that each $q_j$ contains no $s_{i+1}$-letter. By Lemma \ref{powerOfGeneratorLength}, $P_j=s_i^{\ell_j}$ for some $|\ell_j|\leq 2^{|P_j|}$, thus the word $w_i=q_0s_i^{\ell_1}q_1s_i^{\ell_2}\cdots q_{m-1}s_i^{\ell_k}q_{k}$ has length $$|q_0|+\sum_{j=1}^k\left(|q_j|+2^{|P_j|}\right )\leq  2^{|w_{i+1}|}\leq 2^{E(m-i-1,n)}=E(m-i,n).$$ Since $w_i=w_i'$ in $G_m$, $w$ and $w_i$ are conjugate via $\gamma_{i+1}\gamma^{-1}$, which has length at most $$|\gamma_{i+1}|+|\gamma|\leq |\gamma_{i+1}|+|w_{i+1}|\leq (m-i-1)E(m-i-2,n)+E(m-i-1,n)\leq(m-i)E(m-i-1,n).$$
     \end{proof}
     
     \begin{lemma}\label{conjugateEntirelyInTwoConsecutiveGenerators}
         Let $u$ and $v$ be words on $\{s_0,\ldots, s_i\}$ such that $u$ and $v$ represent elements of $\langle s_i,s_{i-1}\rangle_{G_m}$ which are conjugate in $\langle s_0,\ldots, s_i\rangle_{G_m}$. Then $c(u,v)\leq \lambda( |u|+|v|)$ for some constant $\lambda>0$.
     \end{lemma}
     \begin{proof}
         Applying Lemma \ref{retractionToLowerGroups}, we see $u$ and $v$ are conjugate in $\langle s_i,s_{i-1}\rangle_{G_m}$. The claim then follows by fact that $\CL_{\BS(1,2)}(n)\simeq n$ \cite{sale2016conjugacy}. 
     \end{proof}
    \begin{prop}\label{GmUpperBound}
    We have $\CL_{G_m}(n)\preceq E(m-1,n)$.      
    \end{prop}
\begin{proof}
    Suppose $u$ and $v$ are words on the generators of $G_m$ such that $u$ and $v$ represent conjugate elements of $G_m$ and $|u|+|v|\leq n$. Let $r=\rk(u)=\rk(v)$. By Lemma \ref{conjugatingToRank}, we can conjugate $u,v$ to some $u_r,v_r\in G_m$ respectively such that $|u_r|,|v_r|\leq E(m-r,n)$,  using conjugators of length at most $(m-r)E(m-r-1,n)\leq mE(m-1,n)$. To compute $c(u_r,v_r)$, let $\Delta$ be a reduced annular diagram witnessing the conjugacy of $u_r$ and $v_r$, and let $k$ be the largest natural number such that $s_k$ appears in $\Delta$. By Lemma \ref{largestIndexInDiagram}, $k= r+1$ or $k=r$, so we have two cases:
\newline 

Case 1. Suppose $r=k$. If $r\leq 1$ then our claim follows from $\CL_{\langle s_0,s_1\rangle_{G_m}}(n)\simeq\CL_{\BS(1,2)}(n)\simeq n$ \cite{sale2016conjugacy} and $\CL_{\langle s_0\rangle_{G_m}}(n)\simeq \CL_\ZZ(n)\simeq n$. For $r>1$, every $s_r$-letter in $u_r$ and $v_r$ is contained in either an $s_r$-arch or a radial $s_r$-corridor. Indeed there exists at least one radial $s_r$-corridor, since otherwise the $s_r$-arches may be excised to give $\rk(u_r)=\rk(v_r)\leq r-1$. If the subwords of $u$ and $v$ between each pair of consecutive radial corridors represent elements of $\langle s_{r-1}\rangle_{G_m}$, then our claim follows from Lemma \ref{conjugateEntirelyInTwoConsecutiveGenerators}, setting $i=r$. 
    
    If not, let $\Delta'$ be the contractible subdiagram of $\Delta$ contained between two consecutive radial $s_r$-corridors, where at least one of the corresponding subwords $w_1$ and $w_2$ of $u$ and $v$ (respectively) does not represent an element of $\langle s_{r-1}\rangle_{G_m}$ (see Figure \ref{BetweenRadialsrCorridors}). Since the two corridors bounding $\Delta'$ are consecutive, every $s_r$-edge in $w_1 $ and $w_2$ is part of an $s_r$-arch in $\Delta$, so for $j=1,2$ there is a set $\mathcal{P}_j$ of $r$-pinches in $w_j$ (\textit{not} cyclic conjugates of $w_j$) such that every $s_{r}$-letter is involved in or an endpoint of some $p\in \mathcal{P}_j$, and, for all $P_1,P_2\in \mathcal{P}_j$, one is involved in the other or they have no letters in common. The set of maximal elements of $\mathcal{P}_j$ have total length at most $|w_j|\leq E(m-r,n)$, so by Lemma \ref{powerOfGeneratorLength} we may replace $w_j$ with a word on $\{s_0,\ldots, s_{r-1}\}$ of length at most $E(m-r+1,n)$, as in the proof of Lemma \ref{conjugatingToRank}. Since $\Delta'$ is reduced and $s_{r-1}$ has infinite order in $G_m$, Lemma \ref{reducedAlongCorrs} implies that  $\Delta'$ contains no contractible $s_r$-rings. Also, the boundary of $\Delta'$ has no $s_r$-letters, so this implies $\Delta'$ has no $s_r$-edges, meaning every $s_{r-1}$-edge is part of an $s_{r-1}$-corridor. Call the $s_{r-1}$-corridors in $\Delta'$ without an endpoint in either $w_1$ or $w_2$ \textit{side-corridors}. If there are no side-corridors, every $s_{r-1}$-corridor has an endpoint in $w_1$ or $w_2$, hence the words along the two radial $s_r$-corridors have length at most $|w_1|+|w_2|\leq 2E(m-r+1,n)$ and our proposition holds.

    Otherwise we have the situation depicted in Figure \ref{BetweenRadialsrCorridors}. In this figure, $|e_j|$ is the number of $s_{r-1}$-corridors running between $w_j$ and the closest side corridor, counted along the lefthand $s_r$-corridor, and $|e_j'|$ is the number of $s_{r-1}$-corridors running between $w_j$ and the closest side corridor, counted along the righthand $s_r$-corridor; note $|e_j|+|e_j'|\leq |w_j|\leq E(m-r+1,n)$. Also, $w_j'\in \langle s_{r-2}\rangle_{G_m}$ is the word along the side facing $w_j$ of the $s_{r-1}$-corridor closest to $w_j$. By reading along the dotted paths, we have we have $$w_j'=s_{r-1}^{e_j}w_js_{r-1}^{e_j'}.$$ Since one of $w_1,w_2$ does not represent an element of $\langle s_{r-1}\rangle_{G_m}$, one of $w_1'$ and $w_2'$ is a non-trivial power of $s_{r-2}$; since $w_1'$ and $w_2'$ are conjugate by inspection of the Figure, both are therefore nontrivial.

\begin{figure}

\tikzset{every picture/.style={line width=0.75pt}} %set default line width to 0.75pt        

\begin{tikzpicture}[x=0.75pt,y=0.75pt,yscale=-.75,xscale=.75]
%uncomment if require: \path (0,463); %set diagram left start at 0, and has height of 463

%Straight Lines [id:da49202397028386935] 
\draw [line width=1.5]    (50,230) -- (50,223) ;
\draw [shift={(50,220)}, rotate = 90] [color={rgb, 255:red, 0; green, 0; blue, 0 }  ][line width=1.5]    (14.21,-6.37) .. controls (9.04,-2.99) and (4.3,-0.87) .. (0,0) .. controls (4.3,0.87) and (9.04,2.99) .. (14.21,6.37)   ;
%Straight Lines [id:da6319764161462695] 
\draw    (600,230) -- (600,222) ;
\draw [shift={(600,220)}, rotate = 90] [color={rgb, 255:red, 0; green, 0; blue, 0 }  ][line width=0.75]    (10.93,-4.9) .. controls (6.95,-2.3) and (3.31,-0.67) .. (0,0) .. controls (3.31,0.67) and (6.95,2.3) .. (10.93,4.9)   ;
%Shape: Rectangle [id:dp398311833784863] 
\draw  [line width=1.5]  (50,350) -- (100,350) -- (100,400) -- (50,400) -- cycle ;
%Shape: Rectangle [id:dp2296262373802711] 
\draw  [line width=1.5]  (50,300) -- (100,300) -- (100,350) -- (50,350) -- cycle ;
%Shape: Rectangle [id:dp3481299142694072] 
\draw  [line width=1.5]  (50,250) -- (100,250) -- (100,300) -- (50,300) -- cycle ;
%Shape: Rectangle [id:dp5791562599836944] 
\draw  [line width=1.5]  (50,200) -- (100,200) -- (100,250) -- (50,250) -- cycle ;
%Shape: Rectangle [id:dp2861806609326478] 
\draw  [line width=1.5]  (50,150) -- (100,150) -- (100,200) -- (50,200) -- cycle ;
%Shape: Rectangle [id:dp03476549178922772] 
\draw  [line width=1.5]  (550,250) -- (600,250) -- (600,300) -- (550,300) -- cycle ;
%Shape: Rectangle [id:dp05890362982411079] 
\draw  [line width=1.5]  (550,200) -- (600,200) -- (600,250) -- (550,250) -- cycle ;
%Shape: Rectangle [id:dp14344997680095295] 
\draw  [line width=1.5]  (550,150) -- (600,150) -- (600,200) -- (550,200) -- cycle ;
%Shape: Rectangle [id:dp9505021390384891] 
\draw  [line width=1.5]  (550,100) -- (600,100) -- (600,150) -- (550,150) -- cycle ;
%Shape: Rectangle [id:dp6416558565958725] 
\draw  [line width=1.5]  (550,50) -- (600,50) -- (600,100) -- (550,100) -- cycle ;
%Straight Lines [id:da41678283046421627] 
\draw    (320,400) -- (328,400) ;
\draw [shift={(330,400)}, rotate = 180] [color={rgb, 255:red, 0; green, 0; blue, 0 }  ][line width=0.75]    (10.93,-4.9) .. controls (6.95,-2.3) and (3.31,-0.67) .. (0,0) .. controls (3.31,0.67) and (6.95,2.3) .. (10.93,4.9)   ;
%Straight Lines [id:da4038419742427255] 
\draw    (320,50) -- (328,50) ;
\draw [shift={(330,50)}, rotate = 180] [color={rgb, 255:red, 0; green, 0; blue, 0 }  ][line width=0.75]    (10.93,-4.9) .. controls (6.95,-2.3) and (3.31,-0.67) .. (0,0) .. controls (3.31,0.67) and (6.95,2.3) .. (10.93,4.9)   ;
%Straight Lines [id:da5206099437410219] 
\draw    (100,250) -- (550,250) ;
%Straight Lines [id:da36561896844638264] 
\draw    (100,300) -- (550,300) ;
%Straight Lines [id:da13858288906838157] 
\draw    (100,200) -- (550,200) ;
%Straight Lines [id:da8166879793664252] 
\draw    (100,150) -- (550,150) ;
%Shape: Arc [id:dp02651620925615228] 
\draw  [draw opacity=0] (480,400) .. controls (480,400) and (480,400) .. (480,400) .. controls (480,344.77) and (511.05,300) .. (549.35,300) -- (549.35,400) -- cycle ; \draw   (480,400) .. controls (480,400) and (480,400) .. (480,400) .. controls (480,344.77) and (511.05,300) .. (549.35,300) ;  
%Shape: Arc [id:dp47232357046280726] 
\draw  [draw opacity=0] (520.45,400) .. controls (520.09,398.33) and (519.9,396.56) .. (519.9,394.74) .. controls (519.9,383.97) and (526.61,375.24) .. (534.9,375.24) .. controls (541.4,375.24) and (546.93,380.61) .. (549.01,388.12) -- (534.9,394.74) -- cycle ; \draw   (520.45,400) .. controls (520.09,398.33) and (519.9,396.56) .. (519.9,394.74) .. controls (519.9,383.97) and (526.61,375.24) .. (534.9,375.24) .. controls (541.4,375.24) and (546.93,380.61) .. (549.01,388.12) ;  
%Shape: Arc [id:dp27642806497195926] 
\draw  [draw opacity=0] (549.78,61.46) .. controls (547.78,69.19) and (542.17,74.76) .. (535.55,74.76) .. controls (527.27,74.76) and (520.55,66.03) .. (520.55,55.26) .. controls (520.55,53.44) and (520.74,51.67) .. (521.1,50) -- (535.55,55.26) -- cycle ; \draw   (549.78,61.46) .. controls (547.78,69.19) and (542.17,74.76) .. (535.55,74.76) .. controls (527.27,74.76) and (520.55,66.03) .. (520.55,55.26) .. controls (520.55,53.44) and (520.74,51.67) .. (521.1,50) ;  
%Shape: Arc [id:dp15382614947527418] 
\draw  [draw opacity=0] (100.16,388.78) .. controls (102.1,380.93) and (107.76,375.24) .. (114.45,375.24) .. controls (122.73,375.24) and (129.45,383.97) .. (129.45,394.74) .. controls (129.45,396.56) and (129.26,398.33) .. (128.9,400) -- (114.45,394.74) -- cycle ; \draw   (100.16,388.78) .. controls (102.1,380.93) and (107.76,375.24) .. (114.45,375.24) .. controls (122.73,375.24) and (129.45,383.97) .. (129.45,394.74) .. controls (129.45,396.56) and (129.26,398.33) .. (128.9,400) ;  
%Shape: Arc [id:dp3821523530061435] 
\draw  [draw opacity=0] (550,100) .. controls (550,100) and (550,100) .. (550,100) .. controls (492.73,100) and (446.31,77.61) .. (446.31,50) -- (550,50) -- cycle ; \draw   (550,100) .. controls (550,100) and (550,100) .. (550,100) .. controls (492.73,100) and (446.31,77.61) .. (446.31,50) ;  
%Shape: Arc [id:dp45885517051156754] 
\draw  [draw opacity=0] (550,150) .. controls (550,150) and (550,150) .. (550,150) .. controls (550,150) and (550,150) .. (550,150) .. controls (466.14,150) and (398.15,105.23) .. (398.15,50) -- (550,50) -- cycle ; \draw   (550,150) .. controls (550,150) and (550,150) .. (550,150) .. controls (550,150) and (550,150) .. (550,150) .. controls (466.14,150) and (398.15,105.23) .. (398.15,50) ;  
%Shape: Arc [id:dp47461826148944786] 
\draw  [draw opacity=0] (101.85,350) .. controls (159.11,350) and (205.54,372.39) .. (205.54,400) -- (101.85,400) -- cycle ; \draw   (101.85,350) .. controls (159.11,350) and (205.54,372.39) .. (205.54,400) ;  
%Shape: Arc [id:dp44211905503949334] 
\draw  [draw opacity=0] (101.85,300) .. controls (101.85,300) and (101.85,300) .. (101.85,300) .. controls (101.85,300) and (101.85,300) .. (101.85,300) .. controls (185.71,300) and (253.69,344.77) .. (253.69,400) -- (101.85,400) -- cycle ; \draw   (101.85,300) .. controls (101.85,300) and (101.85,300) .. (101.85,300) .. controls (101.85,300) and (101.85,300) .. (101.85,300) .. controls (185.71,300) and (253.69,344.77) .. (253.69,400) ;  

%Straight Lines [id:da5873825493777933] 
\draw [line width=1.5]    (100.65,150) -- (100.65,50) ;
%Straight Lines [id:da1162585358285757] 
\draw    (550,300) -- (549.35,400) ;
%Shape: Arc [id:dp2818288126429672] 
\draw  [draw opacity=0] (170,50) .. controls (170,105.23) and (138.95,150) .. (100.65,150) .. controls (100.65,150) and (100.65,150) .. (100.65,150) -- (100.65,50) -- cycle ; \draw   (170,50) .. controls (170,105.23) and (138.95,150) .. (100.65,150) .. controls (100.65,150) and (100.65,150) .. (100.65,150) ;  
%Shape: Arc [id:dp8004300190509477] 
\draw  [draw opacity=0] (129.55,50) .. controls (129.91,51.67) and (130.1,53.44) .. (130.1,55.26) .. controls (130.1,66.03) and (123.39,74.76) .. (115.1,74.76) .. controls (108.54,74.76) and (102.96,69.28) .. (100.93,61.65) -- (115.1,55.26) -- cycle ; \draw   (129.55,50) .. controls (129.91,51.67) and (130.1,53.44) .. (130.1,55.26) .. controls (130.1,66.03) and (123.39,74.76) .. (115.1,74.76) .. controls (108.54,74.76) and (102.96,69.28) .. (100.93,61.65) ;  
%Straight Lines [id:da7315919371089235] 
\draw  [dash pattern={on 0.84pt off 2.51pt}]  (110,310) -- (540,310) ;
%Straight Lines [id:da7543942584465912] 
\draw  [dash pattern={on 0.84pt off 2.51pt}]  (110,140) -- (540,140) ;
%Straight Lines [id:da09250698849494166] 
\draw    (320,140) -- (328,140) ;
\draw [shift={(330,140)}, rotate = 180] [color={rgb, 255:red, 0; green, 0; blue, 0 }  ][line width=0.75]    (10.93,-4.9) .. controls (6.95,-2.3) and (3.31,-0.67) .. (0,0) .. controls (3.31,0.67) and (6.95,2.3) .. (10.93,4.9)   ;
%Straight Lines [id:da3567134372881118] 
\draw    (320,310) -- (328,310) ;
\draw [shift={(330,310)}, rotate = 180] [color={rgb, 255:red, 0; green, 0; blue, 0 }  ][line width=0.75]    (10.93,-4.9) .. controls (6.95,-2.3) and (3.31,-0.67) .. (0,0) .. controls (3.31,0.67) and (6.95,2.3) .. (10.93,4.9)   ;
%Straight Lines [id:da3033155024043086] 
\draw  [dash pattern={on 0.84pt off 2.51pt}]  (110,140) -- (110,60) ;
%Straight Lines [id:da4260914969717089] 
\draw  [dash pattern={on 0.84pt off 2.51pt}]  (110,60) -- (540,60) ;
%Straight Lines [id:da7063716507279976] 
\draw    (320,60) -- (328,60) ;
\draw [shift={(330,60)}, rotate = 180] [color={rgb, 255:red, 0; green, 0; blue, 0 }  ][line width=0.75]    (10.93,-4.9) .. controls (6.95,-2.3) and (3.31,-0.67) .. (0,0) .. controls (3.31,0.67) and (6.95,2.3) .. (10.93,4.9)   ;
%Straight Lines [id:da8126714984712817] 
\draw    (110,100) -- (110,92) ;
\draw [shift={(110,90)}, rotate = 90] [color={rgb, 255:red, 0; green, 0; blue, 0 }  ][line width=0.75]    (10.93,-4.9) .. controls (6.95,-2.3) and (3.31,-0.67) .. (0,0) .. controls (3.31,0.67) and (6.95,2.3) .. (10.93,4.9)   ;
%Straight Lines [id:da6303412835921574] 
\draw  [dash pattern={on 0.84pt off 2.51pt}]  (540,140) -- (540,60) ;
%Straight Lines [id:da9421536886962465] 
\draw    (540,90) -- (540,98) ;
\draw [shift={(540,100)}, rotate = 270] [color={rgb, 255:red, 0; green, 0; blue, 0 }  ][line width=0.75]    (10.93,-4.9) .. controls (6.95,-2.3) and (3.31,-0.67) .. (0,0) .. controls (3.31,0.67) and (6.95,2.3) .. (10.93,4.9)   ;
%Straight Lines [id:da9788605689234995] 
\draw  [dash pattern={on 0.84pt off 2.51pt}]  (110,390) -- (110,310) ;
%Straight Lines [id:da5204960910797669] 
\draw    (110,340) -- (110,348) ;
\draw [shift={(110,350)}, rotate = 270] [color={rgb, 255:red, 0; green, 0; blue, 0 }  ][line width=0.75]    (10.93,-4.9) .. controls (6.95,-2.3) and (3.31,-0.67) .. (0,0) .. controls (3.31,0.67) and (6.95,2.3) .. (10.93,4.9)   ;
%Straight Lines [id:da32801480664605753] 
\draw  [dash pattern={on 0.84pt off 2.51pt}]  (540,390) -- (540,310) ;
%Straight Lines [id:da9277627422334258] 
\draw    (540,350) -- (540,342) ;
\draw [shift={(540,340)}, rotate = 90] [color={rgb, 255:red, 0; green, 0; blue, 0 }  ][line width=0.75]    (10.93,-4.9) .. controls (6.95,-2.3) and (3.31,-0.67) .. (0,0) .. controls (3.31,0.67) and (6.95,2.3) .. (10.93,4.9)   ;
%Straight Lines [id:da5893851991771621] 
\draw  [dash pattern={on 0.84pt off 2.51pt}]  (110,390) -- (540,390) ;
%Straight Lines [id:da1674896679725666] 
\draw    (320,390) -- (328,390) ;
\draw [shift={(330,390)}, rotate = 180] [color={rgb, 255:red, 0; green, 0; blue, 0 }  ][line width=0.75]    (10.93,-4.9) .. controls (6.95,-2.3) and (3.31,-0.67) .. (0,0) .. controls (3.31,0.67) and (6.95,2.3) .. (10.93,4.9)   ;
%Straight Lines [id:da7320872555379677] 
\draw [line width=1.5]    (50,150) -- (50,50) ;
%Straight Lines [id:da1921400396163121] 
\draw [line width=1.5]    (100,50) -- (50,50) ;
%Straight Lines [id:da667908478565259] 
\draw [line width=1.5]    (600,400) -- (600,300) ;
%Straight Lines [id:da42228509587302865] 
\draw [line width=1.5]    (549.35,400) -- (549.35,300) ;
%Straight Lines [id:da4038484251282851] 
\draw [line width=1.5]    (600,400) -- (550,400) ;
%Straight Lines [id:da8688357022920112] 
\draw    (100,50) -- (550,49.94) ;
%Straight Lines [id:da40479592513704565] 
\draw    (100,400) -- (550,399.94) ;

% Text Node
\draw (69.3,132.4) node [anchor=north west][inner sep=0.75pt]    {$s_{r}$};
% Text Node
\draw (69.3,182.4) node [anchor=north west][inner sep=0.75pt]    {$s_{r}$};
% Text Node
\draw (71,232.4) node [anchor=north west][inner sep=0.75pt]    {$s_{r}$};
% Text Node
\draw (69.3,282.4) node [anchor=north west][inner sep=0.75pt]    {$s_{r}$};
% Text Node
\draw (69.3,332.4) node [anchor=north west][inner sep=0.75pt]    {$s_{r}$};
% Text Node
\draw (69.3,382.4) node [anchor=north west][inner sep=0.75pt]    {$s_{r}$};
% Text Node
\draw (567,82.4) node [anchor=north west][inner sep=0.75pt]    {$s_{r}$};
% Text Node
\draw (567,132.4) node [anchor=north west][inner sep=0.75pt]    {$s_{r}$};
% Text Node
\draw (567,182.4) node [anchor=north west][inner sep=0.75pt]    {$s_{r}$};
% Text Node
\draw (567,232.4) node [anchor=north west][inner sep=0.75pt]    {$s_{r}$};
% Text Node
\draw (567,282.4) node [anchor=north west][inner sep=0.75pt]    {$s_{r}$};
% Text Node
\draw (567,382.4) node [anchor=north west][inner sep=0.75pt]    {$s_{r}$};
% Text Node
\draw (311,60.4) node [anchor=north west][inner sep=0.75pt]    {$w_{1}$};
% Text Node
\draw (311,362.4) node [anchor=north west][inner sep=0.75pt]    {$w_{2}$};
% Text Node
\draw (311,110.4) node [anchor=north west][inner sep=0.75pt]    {$w_{1} '$};
% Text Node
\draw (311,320.4) node [anchor=north west][inner sep=0.75pt]    {$w_{2} '$};
% Text Node
\draw (112,103.4) node [anchor=north west][inner sep=0.75pt]    {$s_{r-1}^{e_{1}}$};
% Text Node
\draw (511,102.4) node [anchor=north west][inner sep=0.75pt]    {$s_{r-1}^{e_{1} '}$};
% Text Node
\draw (511,313.4) node [anchor=north west][inner sep=0.75pt]    {$s_{r-1}^{e_{2} '}$};
% Text Node
\draw (121,313.4) node [anchor=north west][inner sep=0.75pt]    {$s_{r-1}^{e_{2}}$};

\end{tikzpicture}

    \caption{Region between two radial $s_r$-corridors where not all side-corridors go to $w_1$ or $w_2$}
    \label{BetweenRadialsrCorridors}
\end{figure}

    Let $T$ be total number of side-corridors, and for $j=1,2$ let $E_j$ be the number of $s_{r-1}$-letters in $w_j$ and $w_j'=s_{r-2}^{\ell_j}$. By Lemma \ref{powerOfGeneratorLength}, $|\ell_j|\leq |w_j|2^{|e_j|+|e_j'|+|E_j|}$. In turn, we have $$|w_j|2^{|e_j|+|e_j'|+|E_j|}\leq 2E(m-r+1,n)2^{4E(m-r+1,n)}\leq E(m-r+2,n)^5.$$ By the inspection of Figure \ref{BetweenRadialsrCorridors}, $s_{r-1}^Tw_1's_{r-1}^{-T}=w_2'$, so $\ell_12^T=\ell_2$, which implies $|T|\leq \max\{\ell_1,\ell_2\}$. By Lemma \ref{logarithmicShrinking} $|s_{r-1}^T|_{G_m}\leq 5C\log_2(|T|)\leq 5C\log_2(\max\{\ell_1,\ell_2\})$. This is in turn bounded above by  $$5C\log_2( E(m-r+2,n))\leq 5C\log_2(E(m,n)).$$ After cyclically conjugating $u$ and $v$ if necessary at a cost of at most $|u_r|+|v_r|\leq mE(m-1,n)$, we have $s_{r-1}^{-e_2+T+e_1}u_rs_{r-1}^{e_2-T-e_1}=v_r$, whence our claim follows by the bounds on $|e_1|+|e_2|$ and $|s_{r-1}^T|_{G_m}$.

    Case 2. Suppose $k=r+1$. Then there must exist some non-contractible $s_k$-ring in $\Delta$. Let $Q_1$ be the $s_k$-ring closest to $u_r$ and $Q_2$ the $s_k$-ring closest to $v_r$, and let $w_1,w_2$ be the words along the sides of $Q_1,Q_2$ closest to $u_r$ and $v_r$ respectively. By Lemma \ref{reducedAlongCorrs}, $w_1$ and $w_2$ are reduced words on $\{s_r\}$. If $r=0$, then $w_1=u_0,w_2=v_0$ since there are no 2-cells between  $Q_1$ and $u_r$ or between $Q_2$ and $v_r$. If $r\geq 1$, then $|w_1|\leq |u_r|$ and $|w_2|\leq |v_r|$ because (for $j=1,2$) $w_j$ being reduced implies every $s_r$-letter in $w_j$ is part of an $s_r$-corridor running between $Q_j$ and the corresponding boundary component of $\Delta$. After cyclically conjugating $u_r$ so that its first letter is an endpoint of an $s_r$-corridor running to $Q_1$, and doing the same to $v_r$ so that its first letter is part of an $s_r$-corridor running to $Q_2$, we have that $u_r$ and $v_r$ are conjugate to $w_1$ and $w_2$ respectively by a power of $s_{r-1}$, hence $u_r,v_r\in \langle s_r,s_{r-1}\rangle_{G_m}$. Lemma \ref{conjugateEntirelyInTwoConsecutiveGenerators} thus implies $c(u_r,w_1),c(v_r,w_2)\leq \lambda E(m-1,n)$. Also, the cyclic conjugations can be done at cost at most $|u_r|+|v_r|\leq 2E(m-1,n)$.

\begin{figure}

\tikzset{every picture/.style={line width=0.75pt}} %set default line width to 0.75pt        

\begin{tikzpicture}[x=0.75pt,y=0.75pt,yscale=-.75,xscale=.75]
%uncomment if require: \path (0,455); %set diagram left start at 0, and has height of 455

%Shape: Circle [id:dp5381149656885521] 
\draw  [line width=1.5]  (150,250) .. controls (150,139.54) and (239.54,50) .. (350,50) .. controls (460.46,50) and (550,139.54) .. (550,250) .. controls (550,360.46) and (460.46,450) .. (350,450) .. controls (239.54,450) and (150,360.46) .. (150,250) -- cycle ;
%Shape: Circle [id:dp24666866404743037] 
\draw  [line width=1.5]  (325,250) .. controls (325,236.19) and (336.19,225) .. (350,225) .. controls (363.81,225) and (375,236.19) .. (375,250) .. controls (375,263.81) and (363.81,275) .. (350,275) .. controls (336.19,275) and (325,263.81) .. (325,250) -- cycle ;
%Shape: Circle [id:dp9266413196218034] 
\draw  [line width=0.75]  (277.5,250) .. controls (277.5,209.96) and (309.96,177.5) .. (350,177.5) .. controls (390.04,177.5) and (422.5,209.96) .. (422.5,250) .. controls (422.5,290.04) and (390.04,322.5) .. (350,322.5) .. controls (309.96,322.5) and (277.5,290.04) .. (277.5,250) -- cycle ;
%Shape: Circle [id:dp42876254489631116] 
\draw  [line width=0.75]  (265,250) .. controls (265,203.06) and (303.06,165) .. (350,165) .. controls (396.94,165) and (435,203.06) .. (435,250) .. controls (435,296.94) and (396.94,335) .. (350,335) .. controls (303.06,335) and (265,296.94) .. (265,250) -- cycle ;
%Shape: Ellipse [id:dp998129696008918] 
\draw  [line width=0.75]  (200.56,250) .. controls (200.56,167.46) and (267.46,100.56) .. (350,100.56) .. controls (432.54,100.56) and (499.44,167.46) .. (499.44,250) .. controls (499.44,332.54) and (432.54,399.44) .. (350,399.44) .. controls (267.46,399.44) and (200.56,332.54) .. (200.56,250) -- cycle ;
%Shape: Circle [id:dp06901076794800332] 
\draw  [line width=0.75]  (190,250) .. controls (190,161.63) and (261.63,90) .. (350,90) .. controls (438.37,90) and (510,161.63) .. (510,250) .. controls (510,338.37) and (438.37,410) .. (350,410) .. controls (261.63,410) and (190,338.37) .. (190,250) -- cycle ;
%Shape: Circle [id:dp12324337560256637] 
\draw  [dash pattern={on 0.84pt off 2.51pt}][line width=0.75]  (282.5,250) .. controls (282.5,212.72) and (312.72,182.5) .. (350,182.5) .. controls (387.28,182.5) and (417.5,212.72) .. (417.5,250) .. controls (417.5,287.28) and (387.28,317.5) .. (350,317.5) .. controls (312.72,317.5) and (282.5,287.28) .. (282.5,250) -- cycle ;
%Shape: Circle [id:dp28258828155883864] 
\draw  [dash pattern={on 0.84pt off 2.51pt}][line width=0.75]  (180,250) .. controls (180,156.11) and (256.11,80) .. (350,80) .. controls (443.89,80) and (520,156.11) .. (520,250) .. controls (520,343.89) and (443.89,420) .. (350,420) .. controls (256.11,420) and (180,343.89) .. (180,250) -- cycle ;
%Straight Lines [id:da5125918052751103] 
\draw    (345,100) -- (345,165) ;
%Straight Lines [id:da27636455561697004] 
\draw    (355,100) -- (355,165) ;

%Straight Lines [id:da14372264833280735] 
\draw    (373.2,101.9) -- (364.19,166.44) ;
%Straight Lines [id:da9491429348135255] 
\draw    (384.2,105.4) -- (374.15,168.81) ;
%Straight Lines [id:da41750160442322337] 
\draw    (400.7,109.9) -- (385,171.64) ;
%Straight Lines [id:da6630193327957037] 
\draw    (411.2,113.9) -- (394.72,176.87) ;
%Shape: Circle [id:dp35367783092731975] 
\draw  [dash pattern={on 0.84pt off 2.51pt}][line width=0.75]  (203.75,250) .. controls (203.75,169.23) and (269.23,103.75) .. (350,103.75) .. controls (430.77,103.75) and (496.25,169.23) .. (496.25,250) .. controls (496.25,330.77) and (430.77,396.25) .. (350,396.25) .. controls (269.23,396.25) and (203.75,330.77) .. (203.75,250) -- cycle ;
%Shape: Circle [id:dp42450377525979544] 
\draw  [dash pattern={on 0.84pt off 2.51pt}][line width=0.75]  (261.88,250) .. controls (261.88,201.33) and (301.33,161.87) .. (350,161.87) .. controls (398.67,161.87) and (438.13,201.33) .. (438.13,250) .. controls (438.13,298.67) and (398.67,338.12) .. (350,338.12) .. controls (301.33,338.12) and (261.88,298.67) .. (261.88,250) -- cycle ;

% Text Node
\draw (409.29,147.82) node [anchor=north west][inner sep=0.75pt]  [rotate=-25.96]  {$\cdots $};
% Text Node
\draw (341,190) node [anchor=north west][inner sep=0.75pt]  [font=\small]  {$s_{r}^{2\ell }$};
% Text Node
\draw (341,55) node [anchor=north west][inner sep=0.75pt]  [font=\small]  {$s_{r}^{2\ell }$};
% Text Node
\draw (455,175) node [anchor=north west][inner sep=0.75pt]  [font=\small]  {$s_{r}^{\ell }$};
% Text Node
\draw (440,210) node [anchor=north west][inner sep=0.75pt]  [font=\small]  {$s_{r}^{\ell }$};

\end{tikzpicture}
\caption{Two non-contractible $s_k$-corridors whose opposite sides have the same label.}

    \label{oppSidesSame}
\end{figure}
    
    It remains to give an upper bound for $c(w_1,w_2)$. All $s_k$-corridors in $\Delta$ are reduced, so in the annular subdiagram $\Delta'$ contained between two consecutive non-contractible $s_k$-rings, there must only be radial $s_{r}$-corridors, hence the words along the boundary components of $\Delta'$ must be the same. This implies that $\Delta'$ can be excised and the two boundary components glued together, as discussed in Lemma \ref{exciseTwoCorridors}. After doing so, if two consecutive $s_k$-corridors have opposite orientations, the words along the sides not facing each other will be the same as in Figure \ref{oppSidesSame}, so these corridors may similarly be excised. Hence the only 2-cells between $Q_1$ and $Q_2$ are those belonging to additional non-contractible $s_k$-rings, which share the same orientation and only have edges labeled $s_r$ and $s_k=s_{r+1}$, so $w_1$ and $w_2$ are conjugate via a word on $\{s_r,s_{k}\}$. In particular, perhaps after switching $w_1$ and $w_2$, the standard normal form for $\BS(1,2)$ implies that they are conjugate via $s_r^Us_{k}^Ts_r^{U'}$ where $T$ is the number of $s_k$-corridors (with the sign of $T$ reflecting their orientation) and $U,U'$ are some integers. But, $s_r^{U'}$ commutes with $w_1$ and $s_r^{-U}$ commutes with $w_2$ , so $w_1$ and $w_2$ are also conjugate via $s_k^T$. Since $w_i=s_{r}^{\pm|w_i|}$ and $0<|w_i|\leq E(m-r,n)$, the equality $s_k^{T}w_1s_k^{-T}=w_2$ implies $\pm|w_1|2^T=\pm|w_2|$, which in turn implies $|s_{r+1}^T|_{G_m}\leq C\log_2(\max\{|w_1|,|w_2|\})$, whence the claim follows.

    \end{proof}

\section{Lower Bound for $G_m$}

This section leverages the result \cite[Lemma 2.8(2)]{miasnikov2012power}, that any so-called ``compact binary sum" has a minimal number of terms among all equal binary sums. Our goal is to construct two elements $u_n,v_n$ whose minimal length conjugator is of length at least $E(m-1,|u_n|+|v_n|)$, and we do so by, first, finding for any $k\in \ZZ$ a power series $k=n_0+\sum_{j=0}^p2^{m_j}n_j$ whose number of terms is bounded above by $|s_i^k|_{G_m}$. 
% \begin{lemma}\label{lengthOfMaxIndexGeneratorPower}
%     For all $k\in \ZZ$, $|s_m^k|_{G_m}=k$.
% \end{lemma}
% \begin{proof}
%     As noted in the proof of Lemma \ref{nestedBS}, $G_m$ is an HNN-extension of $G_{m-1}$ with stable letter $s_m$, whence the claim follows.
% \end{proof}
\begin{lemma}\label{numberOfTerms}
% \label{lengthToDyadicBound}
    If $|s_i^k|_{G_m}\leq n$ then there exists some $n_0,n_1,\ldots,n_p,m_1,\ldots,m_p\in \ZZ$ such that $k=n_0+\sum_{j=1}^p2^{m_j}n_j$ and $\sum_{j=0}^p|n_j|\leq n$. 
    
    % is an $(n,i)$-dyadic number. 
\end{lemma}

\begin{proof}
    % We proceed by downward induction on $i$. In the case that $m=i$, the claim is equivalent to the statement that $|s_m^k|_{G_m}\leq n$ implies $|k|\leq n$, which follows immediately from Lemma \ref{lengthOfMaxIndexGeneratorPower}. For the inductive step, 
    
    Let $w$ be a minimum length word representing $s_i^k$. By Corollary \ref{noLowerGeneratorsInMinimumLengthWord}, we may assume it is a word on $\{s_i,\ldots,s_m\}$. Let $\{w_1,\ldots,w_p\}$ be a set of disjoint subwords of $w$ satisfying:
    \begin{enumerate}\label{setOfSubwords}
        \item\label{property1} $w_j\in \langle s_{i+1},\ldots,s_m\rangle_{G_m}$ for all $j$;
        \item\label{property2} no $w_j$ is properly contained in a subword $w'$ of $w$ with $w'\in \langle s_{i+1},\ldots,s_m\rangle_{G_m}$;
        \item\label{property3}for all $j'>i$, every $s_{j'}$-letter of $w$ is contained in some $w_j$.
    \end{enumerate}
    Perhaps after re-indexing, by \eqref{property3} we may write $w=s_i^{n_0}w_1s_i^{n_1}w_2\cdots w_ps_i^{n_p}$ for some $n_j\in \ZZ$ such that \begin{equation}\label{lengthOfW}
        |n_0|+\sum_{j=1}^p(|w_j|+|n_j|)=|s_i^k|_{G_m}.
    \end{equation} Every $w_j$ is equal in $G_m$ to some word $w_j'$ with no $q$-pinches for $q>i$. The word $w'=s_i^{n_0}w_1's_i^{n_1}w_2'\cdots w_p's_i^{n_p}$ equals $w$ in $G_m$ and satisfies items \eqref{property1},\eqref{property2}, and \eqref{property3} listed above. We claim that $w_j'\in \langle s_{i+1}\rangle_{G_m}$, or equivalently, $w'$ contains no $s_{q'}$-letters for $q'>i+1$. Let $q_0$ be the largest index of a letter appearing in $w'$. If $q_0>i+1$, then by  Lemma \ref{higherGeneratorInPinchVanKampen} every $s_{q_0}$-letter must be the endpoint of an $q_0$-pinch of $w'$. Without loss of generality, suppose some $q_0$-pinch has one endpoint in $w_1'$ and the other in $w_2'$, so $w_1'=u_1s_{q_0}^{\pm1}v_1$ and $w_2'=u_2s_{q_0}^{\mp1}v_2$ with $s_{q_0}^{\pm1}v_1s_i^{n_1}u_2s_{q_0}^{\mp1}$ a $q_0$-pinch. Then, $w_1's_i^{n_1}w_2'=u_1s_{q_0-1}^{\ell}v_2$ for some $\ell\in \ZZ$. This is a word on $\{s_{i+1},\ldots,s_m\}$, hence $w_1's_i^{n_1}w_2'\in\langle s_{i+1},\ldots,s_m\rangle_{G_m}$, contradicting item \eqref{property2}. By item \eqref{property1} we thus have $w_j'=s_{i+1}^{\ell_j}$ for some $\ell_j\in \ZZ$. 
    
    Now, $\langle s_i,s_{i+1}\rangle_{G_m}$ is isomorphic to $\BS(1,2)$, wherein an straightforward computation shows that if $$s_i^{n_0}s_{i+1}^{\ell_1}s_i^{n_1}s_{i+1}^{\ell_2}\cdots s_{i+1}^{\ell_p}s_i^{n_p}=s_i^k,$$ then $$k=n_0+\sum_{j=1}^p2^{m_j}n_j,$$ where $$m_j=\sum_{\alpha=1}^{j}\ell_{\alpha}.$$ Our claim follows by equation (\ref{lengthOfW}).

    % and that $$\sum_{j=1}^p\ell_j=0.$$\commentC{I am running out of good index variables...} By the definition of $m_j$, we also have $$m_j=-\sum_{q=j+1}^{p}\ell_{q}.$$ For every $j$, $$\min\left\{\sum_{q=1}^j|w_q|,\sum_{q=j+1}^p|w_q|\right\}\leq\frac{\sum_{q=1}^p|w_q|}{2}.$$ Combining this with Remarks \ref{multByNegativeOne} and \ref{addingNumbers}, we see that $m_j$ is a $(\frac{1}{2}\sum_{q=1}^p|w_q|,i+1)$-dyadic number for all $j$. Since $$n_0+2(\frac{1}{2}\sum_{q=1}^p|w_q|)+\sum_q|n_q|=n_0+\sum_{q=1}^p|w_q|+\sum_q|n_q|,$$ Equation (\ref{lengthOfW}) implies that $k$ is an $(n,i)$-dyadic number as desired.
    \end{proof}

       Now, taking care to still control the number of terms, we convert an arbitrary power series of the above form into a binary sum in the sense of \cite{miasnikov2012power}.\begin{lemma}\label{convertTo MAW BinarySums}
        Let $k=n_0+\sum_{j=0}^p2^{m_j}n_j$, where $n_j\neq 0$ for all $j=0,\ldots, p$. Then $k$ can be expressed as $$k=\varepsilon_0+\sum_{j=1}^{p'}2^{m_j'}\varepsilon_j$$ where $\varepsilon_j\in \{-1,1\}$ for all $j$ and $p'\leq 2\sum_{j=0}^p|n_j|$.
    \end{lemma}

    \begin{proof}
        Let $\lambda$ be the smallest index such that $n_{\lambda}\not\in \{-1,1\}$. Then, $n_{\lambda}=\pm(\varepsilon_{\lambda}+2^kn_{\lambda}')$ for some $\varepsilon_{\lambda}\in\{-1,0,1\}$ and $n_{\lambda}'$ such that $|\varepsilon_{\lambda}|+|n_{\lambda}'|<|n_{\lambda}|$. Replace $n_{\lambda}$ with $2^{m_{\lambda}}\varepsilon_{\lambda}+2^{m_{\lambda}+k}n'_{\lambda}$ in our sum. Since $$|\varepsilon_{\lambda}|+|n_{\lambda}'|+\sum_{j=0}^{\lambda-1}|n_j|+\sum_{j=\lambda+1}^{p}|n_j|\leq \sum_{j=0}^{p}|n_j|,$$ we may repeat this at most $\sum_{j=0}^{p}|n_j|$-times to obtain a sum $k=\varepsilon_0+\sum_{j=0}^{p'}2^{m_j'}\varepsilon_j$ where $\varepsilon_j\in \{-1,1\}$ for all $j$. Each iteration adds at most one term, so $p'\leq p+\sum_{j=0}^p|n_j|\leq 2\sum_{j=0}^p|n_j|$.
    \end{proof}
For $k_0$ equal to $(4(m,n)-1)/(3\cdot 2^\alpha)$ for some $\alpha$, which is a compact binary sum, \cite{miasnikov2012power} then allows us to compute a lower bound for $|s_i^{k_0}|_{G_m}$. 

\begin{cor}\label{lengthOfCompactPower}
        Let $k=\sum_{j=0}^{p}2^{2j+\alpha}$ for some $\alpha\in \ZZ$. Then $|s_i^k|_{G_m}\geq (p-1)/2$.
    \end{cor}
    \begin{proof}
        Let $k=n_0+\sum_{j=1}^{q}2^{m_j}n_j$ be the representation of $k$ with the smallest value of $|n_0|+\sum_{j=1}^{q}|n_j|$ such that $n_j\neq 0$ for all $j$. By Lemma \ref{numberOfTerms}, $\sum_{j=0}^{q}|n_j|\leq |s_i^k|_{G_m}$. By Lemma \ref{convertTo MAW BinarySums}, \begin{equation}\label{compactEq}k=\varepsilon_0+\sum_{j=1}^{q'}2^{m_j'}\varepsilon_j,\end{equation} where $\varepsilon_j\in \{-1,1\}$ for all $j$ and $q'\leq 2\sum_{j=0}^q|n_j|$. By Lemma 2.8(2) of \cite{miasnikov2012power}, every sum equal to $k$ of the form of equation (\ref{compactEq}) has at least $p$ terms. Thus, we have $$p\leq q'+1\leq 2\sum_{j=0}^q|n_j|+1\leq 2|s_i^k|_{G_m}+1,$$ whence our claim follows.
    \end{proof}
Before completing this section, we need two more computational lemmas that will bound the lengths of specific elements appearing in our argument:
\begin{lemma}\label{lengthOfPowersOf2OfGenerators}
    For $i=0,\ldots, m$, $\abs{s_i^{E(m-i,n)}}_{G_m}\leq 2^{m-i}\cdot (n+m-i)$.
\end{lemma}

\begin{proof}
    We proceed by downwards induction on $i$. For $i=m$, the claim holds because $\abs{s_{m}^{n}}_{G_m}\leq n$. For the inductive step, observe $$s_i^{E(m-i,n)}=s_{i+1}^{E(m-i-1,n)}s_is_{i+1}^{-E(m-i-1,n)}.$$ Thus, $$\abs{s_i^{E(m-i,n)}}_{G_m}\leq 2\abs{s_{i+1}^{E(m-i-1,n)}}_{G_m}+1\leq   2^{m-i}(n+m-i)$$ as desired.
\end{proof}

\begin{lemma}\label{s1s0bound}
    Assume $m\geq 2$ and let $\alpha\leq 0$ and $\beta=-2^{-\alpha}\cdot(4E(m,n)-1)/3$. Then $|s_1^\alpha s_0^\beta|_{G_m}\geq (E(m-1,n)-2)/8$.
\end{lemma}
\begin{proof}
 Observe $$3\left(\sum_{j=0}^{E(m-1,n)/2}2^{-\alpha+2j}\right) = \sum_{j=0}^{E(m-1,n)/2}2^{-\alpha+2j}+\sum_{j=0}^{E(m-1,n)/2}2^{-\alpha+2j+1}=\sum_{j'=0}^{E(m-1,n)+1}2^{-\alpha+j'}$$$$=2^{-\alpha}(2^{E(m-1,n)+2}-1)=2^{-\alpha} (4E(m,n)-1),$$ so $$-\beta=\sum_{j=0}^{E(m-1,n)/2}2^{-\alpha+2j}.$$ By Corollary \ref{lengthOfCompactPower}, \begin{equation}\label{betabound}
           |s_0^{-\beta}|_{G_m}\geq (E(m-1,n)-2)/4.\end{equation} Since $|s_0^\beta|_{G_m}=|s_0^{-\beta}|_{G_m}$, combining \eqref{betabound} with Corollary \ref{retractShrinksWords} and the inequality $|s_1^\alpha s_0^{\beta}|_{G_m}|\geq |s_0^\beta|_{G_m}-|s_1^\alpha|_{G_m}$ gives $$|s_1^\alpha s_0^{\beta}|_{G_m}|\geq \max\left\{|s_0^\beta|_{G_m}-|s_1^\alpha|_{G_m},|s_1^{\alpha}|_{G_m}\right\}\geq |s_0^\beta|_{G_m}/2\geq (E(m-1,n)-2)/8 ,$$ whence our claim follows. 
\end{proof}
Now we turn to proving the lower bound in Theorem \ref{FirstTheorem}. The case $m=1$ is trivial, since $\CL_H(n)\succeq n$ for all finitely presented groups $H$. The remainder of the theorem is handled in the following proposition:
\begin{prop}\label{IteratedLowerBound}
    For $m\geq 2$, we have $\CL_{G_m}(n)\succeq E(m-1,n)$.
\end{prop}
    \begin{proof}
        Let $u_n=s_1^2$ and $v_n=s_0^{4E(m,n)-1}s_1^2$, let $\Delta$ be a reduced annular diagram witnessing the conjugacy of $u_n$ and $v_n$, and observe that $u_n,v_n$ both have length linear in $n$ by Lemma \ref{lengthOfPowersOf2OfGenerators}. Let $k$ be the largest natural number such that there is an $s_k$-edge in $\Delta$. By Lemma \ref{largestIndexInDiagram}, $k\leq 2$. If $k=2$, then there is a non-contractible $s_2$-ring in $\Delta$, since there are no radial $s_2$-corridors or $s_2$-arches, and any contractible $s_2$-ring would have a non-trivial power of $s_1$ on its outer boundary by Lemma \ref{reducedAlongCorrs}, contradicting the fact that $s_1$ has infinite order. Let $w$ be the word along the side closest to $v_n$ of the $s_2$-ring $Q$ closest to $v_n$. This is a power $s_1^e$ of $s_1$ conjugate to $u_n$. Moreover, $w$ is conjugate to $v_n$ in $G_1$, since there are no $s_2$-edges between $Q$ and the boundary corresponding to $v_n$. A quick calculation in $\BS(1,2)\simeq G_1$ gives $w=s_1^2=u_n$. By restricting attention to the subdiagram contained between $Q$ and $v_n$, we may assume $k\leq 1$.

        We claim a minimal length conjugator $w$ can be written as a word on $\{s_0,s_1\}$. Indeed, let $\Omega$ be a reduced annular diagram for $u_n,v_n$ -- by Lemma \ref{largestIndexInDiagram}, we may assume $\Omega$ has only $s_0,s_1,s_2$ edges. If it contains an $s_2$-edge, it is part of a non-contractible $s_2$-corridor, and in particular $\Omega$ contains a non-contractible $s_2$-corridor $S$ connected to $v_n$ via two radial $s_1$-corridors. The side facing $v_n$ must therefore be labeled $s_1^2=u_n$, so we may replace $\Omega $ with the sub-annulus between $S$ and $v_n$ to see that there is a strictly shorter conjugator taking $u_n$ to $v_n$.

        In any event, this minimal length conjugator $w$ can therefore be written as $s_1^\alpha s_0^\beta s_1^\gamma$ for $\alpha,-\gamma\leq 0$ and $\beta\in \ZZ$. Yet, $s_1^\gamma$ commutes with $u_n$, so (after, say, deleting all $s_1$-letters from $w$ with positive exponent) the conjugator can be written $s_1^\alpha s_0^\beta $. We have $$s_1^\alpha s_0^\beta s_1^2s_0^{-\beta}s_1 ^{-\alpha}=s_1^\alpha s_0^{\beta}s_0^{-4\beta}s_1^{2}s_1^{-\alpha}=s_1^\alpha s_0^{-3\beta}s_1^{2}s_1^{-\alpha}$$$$=s_0^{-3\beta \cdot 2^\alpha}s_1^\alpha s_1^{2}s_1^{-\alpha}=s_0^{-3\beta\cdot 2^\alpha}s_1^2.$$  So,  $-3\beta\cdot 2^\alpha=4E(m,n)-1$. Lemma \ref{s1s0bound} implies $|s_1^\alpha s_0^\beta|_{G_m}\geq (E(m-1,n)-2)/8$, and we are done.
    \end{proof}
Combined with Proposition \ref{GmUpperBound}, this completes the proof of Theorem \ref{FirstTheorem}.

\section{Upper Bound For The Baumslag-Gersten Group}

This section proceeds first by removing all possible $t$-pinches from a pair of conjugate words (Lemma \ref{cycConjGersten}), controlling the length of the result via an analysis of subgroup distortion due to Platonov (Lemma \ref{GerstenDistortionLemma} and Corollary \ref{GerstenDistortionCor}) -- we refer the reader to \cite{RileyDehn} for a discussion in English on the latter topic. Then, we explicate an algorithm of Diekert-Miasnikov-Wei\ss\  \cite{DMW} to show that, if any $t$-letters remain, conjugator length is linear \textit{in the length of the resulting words} (Lemma \ref{DMWlength}). Proposition \ref{upperBoundGersten} substantially completes the argument, handling the case where no $t$-letters remain via a detailed analysis of annular diagrams over $G$. Some final book-keeping is handled in Corollary \ref{finalGerstenUpper}. Recall that the presentation we are using for $G$ is $$G=\langle s_0,s_1,t\mid s_1s_0s_1^{-1}=s_0^2,\ ts_0t^{-1}=s_1\rangle.$$

\begin{lemma}\label{GerstenDistortionLemma}
    There exists a constant $C_1>0$ such that, for all $n\in \Naturals$, $$|n|\leq C_1\cdot E(\lfloor \log_2 C_1|s_0^n|_G+C_1\rfloor,1)+C_1.$$
\end{lemma}

Since $\abs{|s_1^n|_G-|s_0^n|_G}=\abs{|s_1^n|_G-|t^{-1}s_1^nt|_G}\leq 2$ for all $n\in \ZZ$, we have the following corollary:
\begin{cor}\label{GerstenDistortionCor}
    There exists a constant $C_2>0$ such that, for all $n\in \Naturals$, $$|n|\leq C_2\cdot E(\lfloor \log_2 C_2|s_1^n|_G+C_2\rfloor,1)+C_2.$$
\end{cor}

For ease of notation let $C=\max\{1,C_1,C_2\}$, and observe that $$C_i\cdot E(\lfloor \log_2 C_iN+C_i\rfloor,1)+C_i\leq C\cdot E(\lfloor \log_2 CN+C\rfloor,1)+C$$ for $i=1,2$ and $N\in \Naturals$.

\begin{lemma}\label{cycConjGersten}
    Given a word $w$ on $\{ s_0,s_1,t\}$, there exists a word $w''$ equal in $G$ to a cyclic conjugate  of $w$ such that
    
    \begin{enumerate}[label = (\alph*)]
        \item\label{noTPinch}no cyclic conjugate of $w''$ has any $t$-pinches, and  
        \item \label{lenBound}$|w''|\leq C|w|(E(\lfloor \log_2 C|w|+C\rfloor, 1)+2 )$.
        \end{enumerate}

\end{lemma}
\begin{proof}
    Let $\mathcal{P}$ be a set of $t$-pinches of cyclic conjugates of $w$ such that 
    \begin{enumerate}[label = (\roman*)]
        \item \label{disjoint}if $P_1,P_2\in \mathcal{P}$, then either $P_1$ and $P_2$ are disjoint or one is involved in the other, and
        \item\label{tPartners} for every pair $t_1,t_2$ of $t$-partners, at least one of $t_1$ and $t_2$ is involved in or an endpoint of an element of $\mathcal{P}$.
    \end{enumerate}
    To see that such a set exists, consider a set $\mathcal{Q}$ satisfying \ref{disjoint}. Let  $P=t_1w_0t_2$ be a pinch, with $t_1=t^{\pm1},t_2=t^{\mp1}$, such that neither $t_1$ nor $t_2$ is involved in or an endpoint of an element of $\mathcal{Q}$. Then, every element of $\mathcal{Q}$ is either involved in or disjoint from $P$, meaning that $\mathcal{Q}\cup\{P\}$ also satisfies \ref{disjoint}.

    Now, let $P$ be a maximal element of $\mathcal{P}$, and let $w'$ be the cyclic conjugate of $w$ such that $P$ is a $t$-pinch of $w'$ and its leftmost endpoint is the first letter of $w'$. By \ref{disjoint}, every element of $P$ is a $t$-pinch of $w'$. Let $\{P_1,\ldots, P_k\}$ be the set of maximal elements of $\mathcal{P}$, indexed so that $P_i$ is to the left of $P_{i+1}$ in $w'$. We may write $w'=w_0P_1w_1P_2w_2\cdots P_kw_k$. Note that \begin{equation}\label{totalPandWlength}|w_0|+\sum_{i=1}^k(|P_i|+|w_i|)=|w'|=|w|.\end{equation}

By the definition of $t$-pinch, $P_i=s_{\delta_i}^{n_i}$ in $G$ for some $\delta_i\in\{0,1\},n_i\in \ZZ$. By Lemma \ref{GerstenDistortionLemma} and Corollary \ref{GerstenDistortionCor}, \begin{equation}\label{sizeni}|n_i|\leq C\cdot E(\lfloor \log_2 C|P_i|+C\rfloor,1)+ C \leq C\cdot E(\lfloor \log_2 C|w|+C\rfloor,1)+ C.\end{equation} Replacing each $P_i$ in $w'$ with $s_{\delta_i}^{n_i}$ gives the word $w''=w_0s_{\delta_1}^{n_1}w_1s_{\delta_2}^{n_2}w_2\cdots s_{\delta_k}^{n_k}w_k$. To see item \ref{noTPinch}, property \ref{tPartners} of $\mathcal{P}$ implies that if $t_1$ is a $t$-letter of $w_j$, then every $t$-partner of $t_1$ in $w'$ is an endpoint of or involved in some $P_i$, and thus $t_1$ has no $t$-partners in $w''$. For item \ref{lenBound}, combining equations \eqref{totalPandWlength} and  \eqref{sizeni} with the fact that $k\leq |w|$ gives \begin{align*}|w''|&\leq|w_0|+\sum_{i=1}^k\left[|w_i|+CE(\lfloor \log_2 C|w|+C\rfloor, 1)+C\right]\\&= k(CE(\lfloor \log_2 C|w|+C\rfloor, 1)+C)+ |w_0|+\sum_{i=1}^k|w_i| \\ &\leq C|w|(E(\lfloor \log_2 C|w|+C\rfloor, 1)+2).\end{align*}
\end{proof}

\begin{lemma}\label{DMWlength} Let $w_1,w_2$ be cyclically Britton-reduced words representing conjugate elements of $G\smallsetminus\langle s_0,s_1\rangle_G$.  The length of the conjugator found by applying algorithm in the proof of Theorem 3 of \cite{DMW}  to $w_1$ and $w_2$ is bounded above by $C(|w_1|+|w_2|)$ for a constant $C>0$.
    
\end{lemma}
\begin{proof}
    Up to cyclic conjugation and taking inverses, the authors of \cite{DMW} write $w_1=t^{-1}\gamma_1t^{\epsilon_1}\ldots t^{\epsilon_n}\gamma_n$ and $w_2= t^{-1}\gamma_1't^{\epsilon_1}\ldots t^{\epsilon_n}\gamma_n'$ for $\gamma_i,\gamma_i'\in \langle s_0,s_1\rangle_G$, and show that $w_1,w_2$ are conjugate via $s_0^k$ for some $k\in \ZZ$. Our aim is to bound $k$ in terms of $|w_1|$ and $|w_2|$.

    Using the notation $s_0=a,s_1=t$, and $G_1=\textbf{BS}(1,2)$, they further write,
    
    \begin{displayquote} We have $ta =_{\textbf{BS}(1,2)} a^2t$ and $at^{-1} =_{\textbf{BS}(1,2)} t^{-1}a^2$. This allows [us] to represent all group elements by words of the form $t^{-p}a^rt^q$ with $p, q \in \Naturals$ and $r \in \ZZ$. ... We denote by $\ZZ[1/2] = \{p/2^q \in \mathbb{Q} \mid p, q \in \ZZ\}$ the ring of dyadic fractions. Multiplication by 2 is an automorphism of the underlying additive group and therefore we can define the semi-direct product $\ZZ[1/2] \rtimes \ZZ$ as follows. Elements are pairs $(r, m) \in \ZZ[1/2] × \ZZ.$ The multiplication in $\ZZ[1/2] \rtimes \ZZ$ is defined by $$(r, m) \cdot (s, q) = (r + 2^ms, m + q).$$ ... It is straightforward to show that $a \mapsto (1, 0)$ and $t \mapsto (0, 1)$ defines a canonical isomorphism between $\textbf{BS}(1,2)$ and $\ZZ[1/2] \rtimes \ZZ$ ... \cite[269-270]{DMW} \end{displayquote}

    Returning to our notation, let $\varphi:s_0\mapsto (1,0),s_1\mapsto (0,1)$ be the automorphism they define, and let $x\in \langle s_0,s_1\rangle_G=G_1$, $\varphi(x)=(a,b)$. The standard normal form of $G_1$, referred to at the beginning of the quote, gives $x=s_1^{-b_1}s_0^as_1^{b_2}$, where $|b_1|+|b_2|\leq |x|_{G_1}$. The second coordinate of $\varphi(s_1^{-b_1}s_0^as_1^{b_2})$ is $-b_1+b_2$, which has magnitude $|b|=|-b_1+b_2|\leq |x|_{G_1}$.

    The authors denote $\gamma_1,\gamma_1'$ by $(r,m)=\varphi(\gamma_1),(s,q)=\varphi(\gamma_2)$ respectively, and so we have $|m|+ \log_2(|r|)+|q| +\log_2(|s|)\leq |\gamma_1|_{G_1}+|\gamma_2|_{G_1}\leq |w_1|+|w_2|$. A computation performed in \cite{DMW} shows that $k\in \{-m,q-m\}$ or that $k$ is at most $\log_2(|r|)+\log_2(|s|)$ (respectively corresponding to cases 1 and 3, and case 2, in the proof of Theorem 3). The claim follows.
\end{proof}

Recall that the ``positive" side of a $t$-corridor is the side labeled by a power of $s_0$, and the ``negative" side is that labeled by a power of $s_1$.

\begin{prop}\label{upperBoundGersten}
    We have $\CL_G(n)\preceq nE(\lfloor \log_2 n\rfloor, 1)$.
\end{prop}
\begin{proof}
    Let $n\geq C$, and let $u,v$ be conjugate elements of $G$ with $|u|_G+|v|_G\leq n$. By Lemma \ref{cycConjGersten}, we may assume $u$ and $v$ are represented by cyclically Britton-reduced words $w_1$ and $w_2$ respectively such that $$|w_i|\leq Cn(E(\lfloor \log_2Cn+C\rfloor,1)+1)+n\leq 2CnE(\lfloor \log_2Cn+C\rfloor,1).$$ Let $\Delta$ be a reduced annular diagram witnessing the conjugacy of $w_1$ and $w_2$ in $G$, and observe that every $t$-letter of $w_1$ and $w_2$ must be part of a radial $t$-corridor in $\Delta$.

Since they are cyclically Britton reduced, $w_1,w_2$ are each conjugate to an element of $\langle s_0,s_1\rangle_G$ if and only if they contain no $t$-letters. If one of them \textit{does} contain a $t$-letter, it being cyclically Britton reduced implies that $t$-letter must be part of a radial $t$-corridor in $\Delta$, so the desired bound follows immediately from Lemma \ref{DMWlength}.

\begin{figure}

\tikzset{every picture/.style={line width=0.75pt}} %set default line width to 0.75pt        

\begin{tikzpicture}[x=0.75pt,y=0.75pt,yscale=-.75,xscale=.75]
%uncomment if require: \path (0,455); %set diagram left start at 0, and has height of 455

%Shape: Circle [id:dp5381149656885521] 
\draw  [line width=1.5]  (150,250) .. controls (150,139.54) and (239.54,50) .. (350,50) .. controls (460.46,50) and (550,139.54) .. (550,250) .. controls (550,360.46) and (460.46,450) .. (350,450) .. controls (239.54,450) and (150,360.46) .. (150,250) -- cycle ;
%Shape: Circle [id:dp24666866404743037] 
\draw  [line width=1.5]  (325,250) .. controls (325,236.19) and (336.19,225) .. (350,225) .. controls (363.81,225) and (375,236.19) .. (375,250) .. controls (375,263.81) and (363.81,275) .. (350,275) .. controls (336.19,275) and (325,263.81) .. (325,250) -- cycle ;
%Shape: Circle [id:dp9266413196218034] 
\draw  [line width=0.75]  (277.5,250) .. controls (277.5,209.96) and (309.96,177.5) .. (350,177.5) .. controls (390.04,177.5) and (422.5,209.96) .. (422.5,250) .. controls (422.5,290.04) and (390.04,322.5) .. (350,322.5) .. controls (309.96,322.5) and (277.5,290.04) .. (277.5,250) -- cycle ;
%Shape: Circle [id:dp42876254489631116] 
\draw  [line width=0.75]  (265,250) .. controls (265,203.06) and (303.06,165) .. (350,165) .. controls (396.94,165) and (435,203.06) .. (435,250) .. controls (435,296.94) and (396.94,335) .. (350,335) .. controls (303.06,335) and (265,296.94) .. (265,250) -- cycle ;
%Shape: Ellipse [id:dp998129696008918] 
\draw  [line width=0.75]  (200.56,250) .. controls (200.56,167.46) and (267.46,100.56) .. (350,100.56) .. controls (432.54,100.56) and (499.44,167.46) .. (499.44,250) .. controls (499.44,332.54) and (432.54,399.44) .. (350,399.44) .. controls (267.46,399.44) and (200.56,332.54) .. (200.56,250) -- cycle ;
%Shape: Circle [id:dp06901076794800332] 
\draw  [line width=0.75]  (190,250) .. controls (190,161.63) and (261.63,90) .. (350,90) .. controls (438.37,90) and (510,161.63) .. (510,250) .. controls (510,338.37) and (438.37,410) .. (350,410) .. controls (261.63,410) and (190,338.37) .. (190,250) -- cycle ;
%Shape: Circle [id:dp12324337560256637] 
\draw  [dash pattern={on 0.84pt off 2.51pt}][line width=0.75]  (282.5,250) .. controls (282.5,212.72) and (312.72,182.5) .. (350,182.5) .. controls (387.28,182.5) and (417.5,212.72) .. (417.5,250) .. controls (417.5,287.28) and (387.28,317.5) .. (350,317.5) .. controls (312.72,317.5) and (282.5,287.28) .. (282.5,250) -- cycle ;
%Shape: Circle [id:dp28258828155883864] 
\draw  [dash pattern={on 0.84pt off 2.51pt}][line width=0.75]  (180,250) .. controls (180,156.11) and (256.11,80) .. (350,80) .. controls (443.89,80) and (520,156.11) .. (520,250) .. controls (520,343.89) and (443.89,420) .. (350,420) .. controls (256.11,420) and (180,343.89) .. (180,250) -- cycle ;
%Straight Lines [id:da5125918052751103] 
\draw    (345,100) -- (345,165) ;
%Straight Lines [id:da27636455561697004] 
\draw    (355,100) -- (355,165) ;

%Straight Lines [id:da14372264833280735] 
\draw    (373.2,101.9) -- (364.19,166.44) ;
%Straight Lines [id:da9491429348135255] 
\draw    (384.2,105.4) -- (374.15,168.81) ;
%Straight Lines [id:da41750160442322337] 
\draw    (400.7,109.9) -- (385,171.64) ;
%Straight Lines [id:da6630193327957037] 
\draw    (411.2,113.9) -- (394.72,176.87) ;
%Shape: Circle [id:dp35367783092731975] 
\draw  [dash pattern={on 0.84pt off 2.51pt}][line width=0.75]  (203.75,250) .. controls (203.75,169.23) and (269.23,103.75) .. (350,103.75) .. controls (430.77,103.75) and (496.25,169.23) .. (496.25,250) .. controls (496.25,330.77) and (430.77,396.25) .. (350,396.25) .. controls (269.23,396.25) and (203.75,330.77) .. (203.75,250) -- cycle ;
%Shape: Circle [id:dp42450377525979544] 
\draw  [dash pattern={on 0.84pt off 2.51pt}][line width=0.75]  (261.88,250) .. controls (261.88,201.33) and (301.33,161.87) .. (350,161.87) .. controls (398.67,161.87) and (438.13,201.33) .. (438.13,250) .. controls (438.13,298.67) and (398.67,338.12) .. (350,338.12) .. controls (301.33,338.12) and (261.88,298.67) .. (261.88,250) -- cycle ;

% Text Node
\draw (409.29,147.82) node [anchor=north west][inner sep=0.75pt]  [rotate=-25.96]  {$\cdots $};
% Text Node
\draw (341,190) node [anchor=north west][inner sep=0.75pt]  [font=\small]  {$s_0^\ell$};
% Text Node
\draw (341,55) node [anchor=north west][inner sep=0.75pt]  [font=\small]  {$s_0^\ell$};
% Text Node
\draw (455,175) node [anchor=north west][inner sep=0.75pt]  [font=\small]  {$s_1^\ell$};
% Text Node
\draw (440,210) node [anchor=north west][inner sep=0.75pt]  [font=\small]  {$s_{1}^{\ell }$};

\end{tikzpicture}
\caption{Two non-contractible $t$-rings whose opposite sides have the same label.}

    \label{oppSidesSameGersten}
\end{figure}
            Suppose instead that $\Delta$ has no radial $t$-corridors. We claim $\Delta$ contains at most two non-contractible $t$-rings, and if there are two, the negative sides face away from each other. Suppose there two non-contractible $t$-rings with their negative sides facing each other. Since all the $s_1$-corridors existing between these two $t$-corridors must run from one of them to the other, these $t$-corridors have identical words on their negative sides -- see figure \ref{oppSidesSameGersten} and note the similarity to Figure \ref{oppSidesSame}. This contradiction does not apply in the case of two positive sides facing each other, since $s_0$-edges do not form corridors (and in any event it is not needed for our proof). The words along their positive sides are thus the same, so both may be excised by Lemma \ref{exciseTwoCorridors}. Lastly, there cannot be two consecutive non-contractible $t$-rings with the same orientation, since the $s_1$-corridors leaving the negative side of one cannot run to the positive side of the other. Among any set of three consecutive non-contractible $t$-rings, either two consecutive ones have the same orientation, or two of them have their negative sides facing each other, whence our claim follows. This gives us two cases:
        \begin{enumerate}
            \item Suppose $\Delta$ has no non-contractible $t$-rings. Then $w_1$ and $w_2$ are elements of $\langle s_0,s_1\rangle=\BS(1,2)$ which are conjugate in the same group. The desired bound then follows from the fact that $\CL_{\BS(1,2)}(n)$ is linear \cite{sale2016conjugacy}.
            \item Suppose $\Delta$ has one or two non-contractible $t$-rings, and, if there are two, the negative sides are facing away from each other. Every $t$-corridor has its negative side facing a boundary component of $\Delta$, and hence has length $\ell$ at most $\max\{|w_1|,|w_2|\}$ because each $s_1$-letter is part of a corridor running to the boundary. Moreover, the word $s_1^\ell$ is conjugate itself to either $w_1$ or $w_2$ in $\langle s_0,s_1\rangle$. If $\Delta$ has one non-contractible $t$-ring, the word along the positive side equals $s_0^\ell$ and is also conjugate in $\langle s_0,s_1\rangle$ to one of the words on the boundary of $\Delta$. If $\Delta$ has two non-contractible $t$-rings $Q_1,Q_2$ with their negative sides (labeled $s_1^{\ell_1},s_1^{\ell_2}$ respectively) facing away from each other, the words $s_0^{\ell_1},s_0^{\ell_2}$ along their positive sides are of length at most $\max\{|w_1|,|w_2|\}$ and conjugate in $\langle s_0,s_1\rangle$ to each other -- the situation is somewhat like Figure \ref{oppSidesSameGersten}, save with $s_1$ and $s_0$ switched and all $s_0$-``corridors" degenerate. In the two $t$-corridor case, the claim again follows from the linearity of $\CL_{\BS(1,2)}(n)$ applied to the three conjugate pairs $(w_i,s_1^{\ell_1}),(s_0^{\ell_1},s_0^{\ell_2}),(s_1^{\ell_2},w_j)$ for distinct $i,j\in \{1,2\}$; the one $t$-corridor case follows similarly.
        \end{enumerate}
\end{proof}

\begin{cor}\label{finalGerstenUpper}
   We have $\CL_G(n)\preceq E(\lfloor \log_2 n\rfloor,1)$.
\end{cor}
\begin{proof}
    First, we claim $E( M,1)^i\leq E(M,i)$ for $M\geq 1$. We proceed by induction on $M$, the base case holding because $(2^1)^i=2^i$. For the inductive step, we have $$E(M,1)^i=2^{iE(M-1,1)}\leq 2^{E(M-1,i)}=E(M,i).$$ Now, by Proposition \ref{upperBoundGersten}, $$\CL_G(n)\preceq n E(\lfloor \log_2 n\rfloor,1)\leq E(\lfloor \log_2 n\rfloor,1)^2\leq E(\lfloor \log_2 n\rfloor,2).$$ By the definition of $E$, this is bounded above by $$E(\lfloor \log_2 n\rfloor+1,1)\leq E(\lfloor \log_2 2n\rfloor,1) \preceq E(\lfloor \log_2 n\rfloor,1), $$ and we are done. 
% First, we claim $E(1,M)\leq E(M,1)$ for $M\geq 1$. Proceeding again by induction on $M$, the base case is trivial. For inductive step, we have for all $M>1$, $$2^M\leq 2^{2^{M-1}}\leq 2^{E(M-1,1)}=E(M,1).$$ This implies that $$n\leq 2^{\lfloor \log_2 n\rfloor +1}\leq E(\lfloor \log_2 n\rfloor +1,1)\leq E(\lfloor \log_2 2n\rfloor +1,1).$$ 
    
 % Combined with Proposition \ref{upperBoundGersten}, this gives $$\CL_G\preceq n^2E(\lfloor \log_22 n\rfloor+1, 1)^2\leq E(\lfloor \log_22 n\rfloor+1, 1)^4\leq E(\lfloor \log_22 n\rfloor+1, 4).$$ Since $\lfloor \log_2 2n\rfloor +1\leq \lfloor \log_2 n\rfloor +2$, we are done.
\end{proof}
\begin{remark}
    Mattes and Wei{\ss} in \cite{MattesWeiss} give a constructive algorithm to determine if a word $u$ is conjugate to an element $g$ for any fixed $g$, using in part ideas presented in \cite{DMW}. While we do not explore this here, we hypothesize that an analysis of their algorithm may give an alternative proof of Corollary \ref{finalGerstenUpper}.
\end{remark}

\section{Lower Bound for The Baumslag-Gersten group}

\begin{lemma}\label{numberOfTermsGersten}
% \label{lengthToDyadicBound}
    If $|s_0^k|_{G}\leq n$ then $k=n_0+\sum_{j=1}^p2^{m_j}n_j$ for some $n_0,n_1,\ldots,n_p,m_1,\ldots,m_p\in \ZZ$ such that $\sum_{j=0}^p|n_j|\leq n$. 
    
    % is an $(n,i)$-dyadic number. 
\end{lemma}
\begin{proof}
 By Lemma \ref{lowerWordBoundGersten}, $|s_0^k|_{G_{2M+2}}\leq |s_0^k|_G $, where $M=|s_0^k|_G$. The claim follows from Lemma \ref{numberOfTerms}.
\end{proof}

\begin{lemma}\label{lengthOfCompactPowerGersen}
    Let $k=\sum_{j=0}^{p}2^{2j+\alpha}$ for some $\alpha\in \ZZ$. Then $|s_0^k|_{G}\geq (p-1)/2$.
\end{lemma}
\begin{proof}
    This holds by the same proof as Corollary \ref{lengthOfCompactPower}, replacing the sole application of Lemma \ref{numberOfTerms} with Lemma \ref{numberOfTermsGersten}.
\end{proof}
\begin{lemma}
    \label{lengthOfPowersOf2OfGeneratorsGersten}
    For any $m,n\in \NN$, $\abs{s_0^{E(m,1)}}_{G}\leq 6\cdot 2^m-5$.
\end{lemma}
\begin{proof}
We proceed by induction on $m$, the case $m=1$ being obvious. For the inductive step, observe that $$s_0^{E(m,1)}=s_1^{E(m-1,1)}s_0s_1^{-E(m-1,1)}=ts_0^{E(m-1,1)}t^{-1}s_0ts_0^{-E(m-1,1)}t^{-1}.$$ By induction, this has length at most $2\cdot(6\cdot 2^{m-1}-5)+5=6\cdot 2^m-5$ as desired.
\end{proof}

\begin{cor}\label{sizeLowerBoundGersten}
    Let $\beta=-2^{-\alpha} (4E(m,n)-1)/3$ for some $\alpha\leq 0$. Then $|s_1^\alpha s_0^\beta|_G\geq (E(m-1,n)-2)/8$.
\end{cor}
\begin{proof}
    Let $M=|s_1^\alpha s_0^\beta|_G$. By Lemma \ref{lowerWordBoundGersten}, $M\geq |s_1^\alpha s_0^\beta|_{H_{M+1}}=|s_{M+2}^\alpha s_{M+1}^\beta|_{G_{2M+2}}$. The same argument as Lemma  \ref{s1s0bound} suffices to finish the proof, changing the indices of the $s$-letters and using Lemmas \ref{numberOfTermsGersten},\ref{lengthOfCompactPowerGersen} as needed as needed.
\end{proof}
\begin{prop}\label{lowerBoundGersten}
    We have $\CL_G(n)\succeq E\left(\left\lfloor  \log_2(n)\right\rfloor-1,1\right)\succeq E\left(\left\lfloor  \log_2(n)\right\rfloor,1\right) $. 
\end{prop}
\begin{proof}
    First, it is easy to see $E\left(\left\lfloor \log_2(n)\right\rfloor-1,1\right)\succeq E\left(\left\lfloor  \log_2(n)\right\rfloor,1\right)$. We now turn to the main claim. Let $n$ be given, and let $e=E(\lfloor\log_2(n)\rfloor,1)$. By Lemma \ref{lengthOfPowersOf2OfGeneratorsGersten} $\abs{s_0^e}_G\leq C n$ for some constant $C>0$. Let $u_n=s_1^2$ and $v_n=s_1^2 s_0^{e-1}$. Then $u_n$ and $v_n$ are conjugate via $s_0^{-(e-1)/3}$. To compute $c(u_n,v_n)$, let $\Delta$ be a reduced annular diagram witnessing the conjugacy of $u_n$ and $v_n$. Consider the non-contractible $t$-ring $Q$ in $\Delta$ closest to $v_n$ (if any exist), and let $w$ be the word along the side closest to $v_n$. Then $w$ is either a power of $s_0$ or a power of $s_1$. Since $w$ and $v_n$ are conjugate in $\langle s_0,s_1\rangle_G$, which is isomorphic to $\BS(1,2)$, this implies $w=s_1^2=u_n$. Restricting attention to the subdiagram bounded by $Q$ and $u_n$, our claim thus follows by the same argument as Proposition \ref{IteratedLowerBound}, using Corollary \ref{sizeLowerBoundGersten} in place of Lemma \ref{s1s0bound}.
\end{proof}
Combined with Corollary \ref{finalGerstenUpper}, this completes the proof of Theorem \ref{SecondTheorem}.

\bibliographystyle{alpha}
\bibliography{bib}
%\bibliography{$HOME/Dropbox/Bibliographies/bibli}

\ni  {Conan Gillis} \rule{0mm}{6mm} \\
Department of Mathematics, 310 Malott Hall,  Cornell University, Ithaca, NY 14853, USA 
\\ {cg527@cornell.edu}, \
\href{https://math.cornell.edu/conan-gillis}{http://www.math.cornell.edu/conan-gillis/}
    
\end{document}